\documentclass[twocolumn]{autart}

\usepackage{graphics} 
\usepackage{epsfig} 
\usepackage{graphicx} 
\usepackage{amsmath} 
\usepackage{amssymb} 
\usepackage{stmaryrd}
\usepackage{amsfonts}
\newtheorem{theorem}{Theorem}
\newtheorem{assumption}{Assumption}
\newtheorem{lemma}{Lemma}

\newenvironment{proof}[1][Proof]{\begin{trivlist}
\item[\hskip \labelsep {\bfseries #1}]}{\end{trivlist}}

\begin{document}

\begin{frontmatter}

\title{Robustness of Nonlinear Predictor Feedback Laws to \\ Time- and State-Dependent Delay Perturbations}


\author{Nikolaos Bekiaris-Liberis and Miroslav Krstic}\ead{nbekiari@ucsd.edu and krstic@ucsd.edu}

\address{Department of Mechanical and Aerospace Engineering, University of California, San Diego, La Jolla, CA 92093-0411, USA}            

\maketitle
\thispagestyle{empty}
\pagestyle{empty}

\begin{abstract}
Much recent progress has been achieved for stabilization of linear and nonlinear systems with input delays that are long and dependent on either time or the plant state---provided the dependence is known. In this paper we consider the delay variations as unknown and study robustness of nominal constant-delay predictor feedbacks under delay variations that depend on time and the state. We show that when the delay perturbation and its rate have sufficiently small magnitude, the local asymptotic stability of the closed-loop system, under the nominal predictor-based design, is preserved. For the special case of linear systems, and under only time-varying delay perturbations, we prove robustness of global exponential stability of the predictor feedback when the delay perturbation and its rate are small in any one of four different metrics. We present two examples, one that is concerned with the control of a DC motor through a network  and one of a bilateral teleoperation between two robotic systems. 
\end{abstract}
\end{frontmatter}
\interdisplaylinepenalty=2500 

\section{Introduction}
Networked control systems are present in various engineering applications such as tele-robotics, remote surgery and automotive systems among others. One of the major reasons is that they are advantageous over traditional control systems in terms of flexibility, reliability, maintenance cost etc. \cite{hespanhasurvey}. However, often their performance can be severely degraded in the presence of delays induced by the network \cite{antsaklis}, \cite{teeln}. When the delay is constant and known, the predictor-based controller  compensates the network-induced delay \cite{KARAFYSAMPL}. Yet, the networked-induced delay might be subject to time-varying and state-dependent uncertainties, which, when they are not considered in the control design, not only can degrade the performance of the control system, but can also destabilize the network \cite{basar}. It is therefore crucial to quantify the robustness properties of the constant-delay predictor feedback in the presence of time- and state-dependent delay uncertainties.

Numerous methodologies exist, dealing with the stability or stabilization of nonlinear systems with input \cite{karafyllis finite}, \cite{mazenciss}, \cite{mazencfeed}, \cite{mazmal}, \cite{mazpra}, \cite{mazencnew}, \cite{michiels2}, \cite{niculescubook}, \cite{zhong1}, \cite{zhong2}, \cite{zhong3}, or state \cite{jankovic raz}, \cite{jankovic nonlinear2}, \cite{karafyllisretarded}, \cite{karnew1}, \cite{karnew2},  \cite{pepe1}, \cite{teel1} delays. Predictor-based techniques are developed for the compensation of long actuator delays in linear \cite{artstein}, \cite{bekiaris}, \cite{jankovic forwarding linear}, \cite{jankovic recursive}, \cite{krstic tools}, \cite{nihtila2} or nonlinear \cite{krsticbek}, \cite{bek}, \cite{karpre}, \cite{krstic feed} systems. Among them, \cite{krsticbek}, \cite{krstic varying}, \cite{nihtila2} are dealing with time-varying and \cite{bek} with state-dependent input delays. Predictor feedback has been also successful in designing stabilizing controllers for linear systems with uncertainties either on the delays \cite{bekiariadaptive}, \cite{delphinenew1} or on the plant parameters \cite{nihtila1}, \cite{yildiz} or on both \cite{delphine2}. 

Although, some of the first predictor-based designs for linear, unstable plants, with constant input delays go back to the 1980s \cite{artstein}, a Lyapunov construction has been unavailable until recently \cite{krstic tools}. In addition, a Lyapunov functional is provided in \cite{krstic feed}, which is employed to the stability analysis of the proposed control design. Yet, the robustness properties of the nominal, constant-delay, predictor-based feedback under time- and state-dependent delay perturbations remain unexplored.

We consider forward-complete nonlinear systems that are locally, exponentially stabilizable in the absence of the delay (by a possibly time-varying control law), for which we employ the predictor-based design. The predictor controller is designed assuming constant input delay and using only an estimation of the unmeasured (since the delay is not known) infinite-dimensional actuator state. We prove robustness  of the constant-delay predictor-based feedback, under simultaneous time-varying and state-dependent perturbations on the delay.  Specifically, using the nonlinear infinite-dimensional backstepping transformation, we construct a Lyapunov functional for the closed-loop system that is comprised of the plant, the predictor feedback and the observer for the actuator state. With the constructed Lyapunov functional, we prove that the closed-loop system remains locally asymptotically stable when the perturbation and its rate are small (Section \ref{sectionnonlinear}). 

We also deal with linear systems under time-varying delay perturbations. We show robustness of global exponential stability of the predictor feedback for the cases where the delay perturbation and its rate either have small magnitude, or their $\mathcal{L}_1$ norm is small, or they converge to zero as the time goes to infinity or, finally, have a small moving average for large times (Section \ref{sectionlinear}). Finally, we illustrate the robustness properties of the predictor feedback with two examples. The first one is an example of a DC motor which is controlled through a network. The network induces a delay which is comprised of a known constant part and an unknown time-varying perturbation on this nominal value. In addition, the delay is subject to a state-dependent perturbation that depends on the armature current. The second example is concerned with the bilateral teleoperation between two robotic systems through a network. The network induces a constant nominal delay which is subject to an unknown time-varying perturbation that has a small moving average after a long period of time (Section \ref{secex}).


\textit{Notation}: We use the common definition of class $\mathcal{K}$, $\mathcal{K}_{\infty}$ and $\mathcal{KL}$ functions from \cite{khalil}. For an $n$-vector, the norm $|\cdot|$ denotes the usual Euclidean norm. We say that a function $\zeta$ $:$ $\mathbb{R}_+\times(0,1)$ $\mapsto$ $\mathbb{R}_+$ belongs to class $\mathcal{KC}$ if it is of class $\mathcal{K}$ with respect to its first argument for each value of its second argument and continuous with respect to its second argument. It belongs to class $\mathcal{KC}_{\infty}$ if it is in $\mathcal{KC}$ and also in $\mathcal{K}_{\infty}$ with respect to its first argument. With $\bar{\zeta}$ we denote the inverse of the function $\zeta$ with respect to its first argument for each value of its second argument. 

\section{Robustness to time- and state-dependent delay perturbations for nonlinear systems}
\label{sectionnonlinear}
We consider nonlinear plants of the form
\begin{eqnarray}
\dot{X}(t)=f\left(X(t),U\left(t-\hat{D}-\delta\left(t,X(t)\right)\right)\right),\label{nonlper}
\end{eqnarray}
where $f:C^2\left(\mathbb{R}^n\times\mathbb{R};\mathbb{R}^n\right)$, satisfies $f(0,0)=0$, $\hat{D}>0$, under the nominal, constant-delay predictor feedback given by
\begin{eqnarray}
U(t)=\kappa\left(t+\hat{D},\hat{P}(t)\right),\label{connl}
\end{eqnarray}
where for all $t-\hat{D}\leq\theta\leq t$
\begin{eqnarray}
\hat{P}(\theta)=X(t)+\int_{t-\hat{D}}^{\theta}f\left(\hat{P}(s),U(s)\right)ds,\label{predes1}
\end{eqnarray}
is the estimate of the actual predictor state $P^*(\theta)$, defined for all $t-\hat{D}-\delta(t,X(t))\leq\theta\leq t$ as 
\begin{eqnarray}
P^*(\theta)&=&X(t)+\int_{t-\hat{D}-\delta\left(t,X(t)\right)}^{\theta}\frac{ f\left(P^*(s),U\left(s\right)\right)}{1-G(s)} ds\label{sysstrexnon}\\
G(s)&=&\delta_t\left(\sigma(s),P^*(s)\right)\nonumber\\
&&+\nabla\delta\left(\sigma(s),P^*(s)\right)f\left(P^*(s),U(s)\right),
\end{eqnarray}
where $\sigma$, the actual predicted time (which should be compared with the estimated predicted time $t+\hat{D}$) is 
\begin{eqnarray}
\sigma(\theta)=t+\int_{t-\hat{D}-\delta\left(t,X(t)\right)}^{\theta}\frac{ds}{1-G(s)}.\label{siinonnew}
\end{eqnarray}
An example of such a model together with the predictor-based controller is shown in Fig. \ref{figure1}.

\begin{figure}[t]
\centering
		\includegraphics[width=.475\textwidth]{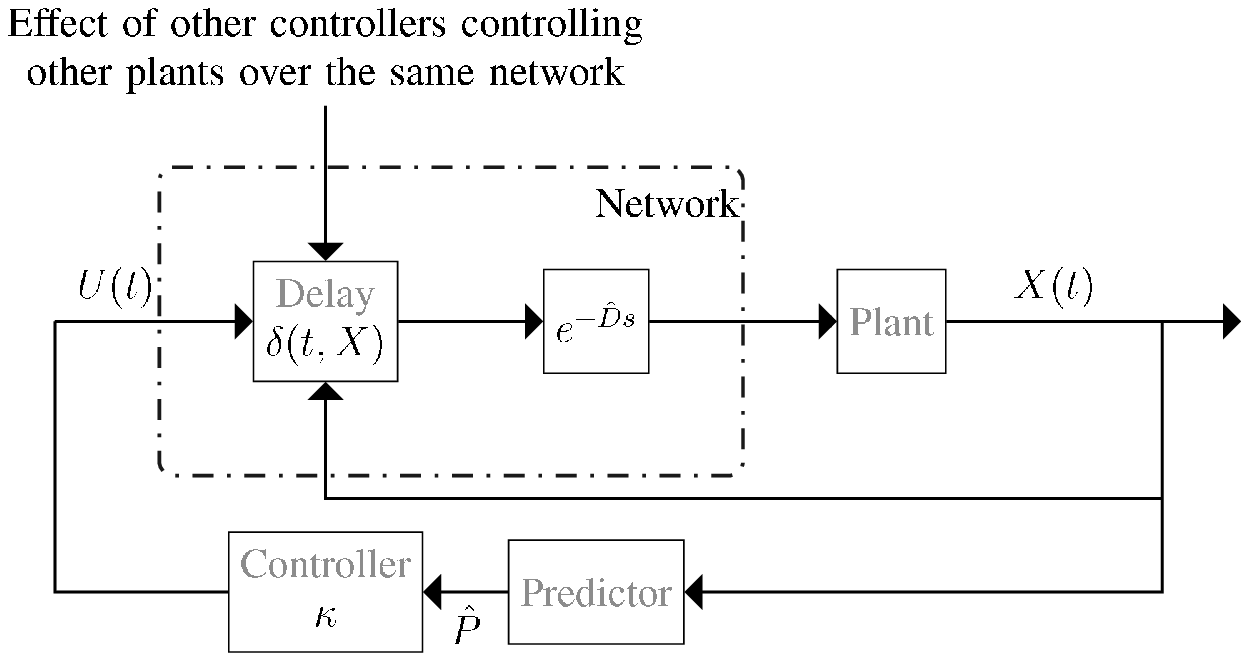}
\caption{Control over a network, with delay that varies with time (as a result of other users's activities) and may be state-dependent. The designer only knows a nominal, constant delay value $\hat D$. The delay fluctuation $\delta(t,X)$ is unknown.}
\label{figure1}
\end{figure}

Let us now make clear why $P^*$ is the actual predictor state of $X$. Define the actual delayed time as
\begin{eqnarray}
\phi(t)&=&t-\hat{D}-\delta(t,X(t))\label{defpnon},
\end{eqnarray}
and the actual predicted time as
\begin{eqnarray}
\sigma(t)&=&\phi^{-1}(t)\nonumber\\
&=&t+\hat{D}+\delta\left(\sigma(t),X(\sigma(t))\right).\label{invpnon} 
\end{eqnarray}
Then we show that the signal in (\ref{sysstrexnon}) satisfies $P^*(t)=X(\sigma(t))$. Differentiating (\ref{invpnon}) we get that 
\begin{eqnarray}
\dot{\sigma}(t)&=&1+\delta_t\left(\sigma(t),X(\sigma(t))\right)\dot{\sigma}(t)\nonumber\\
&&+\nabla\delta\left(\sigma(t),X(\sigma(t))\right)X'(\sigma(t))\dot{\sigma}(t),
\end{eqnarray}
and since $\sigma(t)=\phi^{-1}(t)$, using (\ref{nonlper}) we have
\begin{eqnarray}
\dot{\sigma}(t)&=&\frac{1}{1-F(t)}\label{siinon}\\
F(t)&=&\delta_t\left(\sigma(t),X(\sigma(t))\right)\nonumber\\
&&+\nabla\delta\left(\sigma(t),X(\sigma(t))\right)f\left(X(\sigma(t)),U(t)\right).
\end{eqnarray}
Define the change of variables $t=\sigma(\theta)$, where $\phi(t)\leq\theta\leq t$. With the help of (\ref{siinon}) re-write (\ref{nonlper}) in terms of $\theta$ as
\begin{eqnarray}
\frac{dP^*(\theta)}{d\theta}=\frac{ f\left(P^*(\theta),U\left(\theta\right)\right)}{1-G(\theta)},\quad \mbox{for all $\phi(t)\leq\theta\leq t$}\label{sysstrnon},
\end{eqnarray}
where we substitute $X(\sigma(\theta))$ with $P^*(\theta)$. Integrating (\ref{sysstrnon}) from $\phi(t)$ to $\theta$ we get (\ref{sysstrexnon}) by using the fact that $P^*(\phi(t))=X(\sigma(\phi(t)))=X(t)$. Noting that $\sigma(\phi(t))=t$ we also get (\ref{siinonnew}). A more detailed discussion about definition (\ref{sysstrexnon}) can be found in \cite{bek}.

The predictor state (\ref{predes1}) is the certainty equivalent predictor for system (\ref{nonlper}). This becomes clear by setting $\delta=0$ in (\ref{sysstrexnon}). Note that $\hat{P}(\theta)$ (or $P^*(\theta)$) should be viewed as the output of an operator, parametrized by $t$, acting on $P(s)$ and  $U(s)$, $t-\hat{D}\leq s\leq \theta$ (or $t-\hat{D}-\delta(t,X(t))\leq s\leq \theta$) in the same way that the solution $X(t)$ to an ODE can be viewed as the output of an operator, parametrized by $t_0$, acting on $X(s)$ and the input $U(s)$, $t_0\leq s\leq t$. However, $\hat{P}$ is given implicitly since the plant is nonlinear (for the same reason that the solution $X(t)$ to a nonlinear ODE is given implicitly). In the case of a linear plant $\dot{X}(t)=AX(t)+BU\left(t-\hat{D}\right)$, equation (\ref{predes1}) for the predictor state can be solved explicitly as $\hat{P}(\theta)=e^{A\left(\theta-t+\hat{D}\right)}X(t)+\int_{t-\hat{D}}^{\theta} e^{A\left(\theta-s\right)}BU(s)ds$, and hence, $\hat{P}(t)=e^{A\hat{D}}X(t)+\int_{t-\hat{D}}^t e^{A(t-\theta)}BU(\theta)d\theta$. This is the standard predictor (used for example in \cite{artstein}), which is obtained using the variations of constants formula for the linear ODE satisfied by the plant. An equivalent representation of the signal $\hat{P}(\theta)$ is 
\begin{eqnarray}
\hat{p}(x,t)=X(t)+\hat{D}\int_0^x f\left(\hat{p}(y,t),\hat{u}(y,t)\right),\label{defprnon}
\end{eqnarray}
for all $x\in[0,1]$, where $\hat{u}$, is the estimation of the actuator state $U\left(\theta\right)$, $t-\hat{D}-\delta\left(t,X(t)\right)\leq\theta\leq t$, which satisfies
\begin{eqnarray}
\hat{D}\hat{u}_t(x,t)&=&\hat{u}_x(x,t)\label{obs1non}\\
\hat{u}(1,t)&=&U(t),\label{obs2non}
\end{eqnarray}
that is,
\begin{eqnarray}
\hat{u}(x,t)=U(t+\hat{D}(x-1)),\quad \mbox{for all $x\in[0,1]$}.\label{est11}
\end{eqnarray}
With this definition, $\hat{p}(x,t)$ is the output of an operator, parametrized by $t$, that acts on $\hat{p}(y,t)$ and $\hat{u}(y,t)$, $y\in[0,x]$. With this representation $\hat{p}(1,t)=\hat{P}(t)$.

Note also that from relation (\ref{sysstrexnon}) we see that for $P^*$ to be well-defined the denominator in (\ref{sysstrexnon}) must satisfy the following condition for all $\theta\geq t_0-D\left(X(t_0)\right)$
\begin{eqnarray}
c&>& \delta_t\left(\sigma(\theta),P^*(\theta)\right)\nonumber\\
&&+\nabla\delta\left(\sigma(\theta),P^*(\theta)\right)f\left(P^*(\theta),U(\theta)\right)\label{prhx},
\end{eqnarray}
for $c\in(0,1]$, which is a condition on the perturbation $\delta$, the initial conditions and the solutions of the system. As it turns out later on, this condition is satisfied by appropriately restricting the perturbation $\delta$ and the initial conditions of the plant. We proceed by imposing the following assumptions on the delay-free plant.
\begin{assumption}
\label{ass1}
The plant $\dot{X}=f\left(X,\omega\right)$ is strongly forward complete, that is, there exist a smooth positive definite function $R$ and class $\mathcal{K}_\infty$ functions $\alpha_1$, $\alpha_2$ and $\alpha_3$ such that 
\begin{eqnarray}
\alpha_1\left(\left|X\right|\right)&\leq& R\left(X\right)\leq\alpha_2\left(\left|X\right|\right)\label{forb1}\\
\frac{\partial R\left(X\right)}{\partial X} f\left(X,\omega\right)&\leq&R\left(X\right)+\alpha_3\left(\left|\omega\right|\right)\label{forb2},
\end{eqnarray}
for all $X\in\mathbb{R}^n$ and for all $\omega\in\mathbb{R}$.
\end{assumption}


Assumption \ref{ass1} guarantees that for every initial condition and every measurable locally essentially bounded input signal, the corresponding solution of the system exists for all times. Forward-completeness is a natural requirement for nonlinear plants with input delay. In the absence of this assumption, i.e., when the plant exhibits a finite-escape time, the control signal might reach the plant ``too late". The difference with standard forward-completeness from \cite{angeli} lies in the fact that $R$ in Assumption \ref{ass1} is positive definite, in accordance to the fact that $f(0,0)=0$.
\begin{assumption}
\label{ass2}
There exist positive constants $\mu$, $r^*$, $b$, $\lambda^*$, a function $\hat{\alpha}$ which belongs to class $\mathcal{K}_{\infty}$ and a function $\kappa:C^3\left([t_0,\infty)\times\mathbb{R}^n;\mathbb{R}\right)$ satisfying for all $t\geq t_0$
\begin{eqnarray}
\left|\frac{\partial^{i+j}\kappa\left(t,\xi\right)}{\partial^i t\partial^j \xi}\right|\!&\leq\!\!&\left\{ \begin{array}{ll}\!\hat{\alpha}\left(\left|\xi\right|\right)\!,&\! \!\mbox{$0\!\leq\! i\!\leq3$,\! $j\!=\!0$}\\\! \mu+\hat{\alpha}\left(\left|\xi\right|\right)\!,&\!\!\mbox{$0\!\leq \!i\!\leq3$,\! $ j \!=\!1\ldots 3-i$}\end{array} \!\!\right\} \label{kap}\!\!,
\end{eqnarray}
such that for the plant $\dot{X}(t)=f\left(X(t),\kappa(t,X(t))\right)$ the following holds for all $X(t_0)\in D_{r^*}$,
\begin{eqnarray}
|X(t)|\leq b|X(t_0)| e^{-\lambda^* (t-t_0)},\quad \mbox{for all $t\geq t_0$},
\end{eqnarray} 
where $D_{r^*}=\left\{X\in\mathbb{R}^n | |X|\leq r^*\right\}$.
\end{assumption}
   
\begin{theorem}
\label{thmpernon}
Consider the closed-loop system consisting of the plant (\ref{nonlper}), control law (\ref{connl}), (\ref{predes1}) and observer (\ref{obs1non})--(\ref{obs2non}). Under Assumptions \ref{ass1} and \ref{ass2} there exist
positive constants $c_1$, $c^{**}$, class $\mathcal{K}_{\infty}$ functions $\hat{\mu}$, $\alpha^*$, a class $\mathcal{KC}_{\infty}$ function $\zeta$,
and a class $\mathcal{KL}$ function $\beta$ such that if the perturbation $\delta$ satisfies 
\begin{eqnarray}
\left|\delta\left(t,\xi\right)\right|+\left|\delta_t\left(t,\xi\right)\right|+\left|\nabla\delta\left(t,\xi\right)\right|&\leq&c_1+ \hat{\mu}\left(|\xi|\right)\label{89},
\end{eqnarray}
for all $(t,\xi)\in [t_0,\infty)\times\mathbb{R}^n$, then for all initial conditions which satisfy
\begin{eqnarray}
\Pi(t_0)< c^{**}\label{normth},
\end{eqnarray}
where
\setlength{\arraycolsep}{0pt}\begin{eqnarray}
\Pi(t)&=&\left|X(t)\right|+\int_{t-\hat{D}}^t\alpha^*\left(|U(\theta)|\right)d\theta+\int_{t-\hat{D}}^t\ddot{U}(\theta)^2d\theta\nonumber\\
&&+\int_{t-\hat{D}-\max\left\{0,\delta(t,X(t))\right\}}^t\dot{U}(\theta)^2d\theta,\label{defgadelnon}
\end{eqnarray}\setlength{\arraycolsep}{5pt}it holds that
\begin{eqnarray}
\Pi(t)\leq \beta\left(\zeta\left(\Pi(t_0),\min\left\{r^*,c,\hat{D}\right\}\right),t-t_0\right)\label{bouthm},
\end{eqnarray}
for all $t\geq t_0$ and some $0<c<1$.  
\end{theorem}




We now introduce the backstepping transformation.
\begin{lemma}
\label{lemma1non}
Define the backstepping transformation 
\begin{eqnarray}
\hat{w}(x,t)=\hat{u}(x,t)-\kappa\left(t+\hat{D}x,\hat{p}(x,t)\right),\label{back1non}
\end{eqnarray}
together with its inverse,
\begin{eqnarray}
\hat{u}(x,t)=\hat{w}(x,t)+\kappa\left(t+\hat{D}x,\hat{\rho}(x,t)\right),\label{invback1non}
\end{eqnarray}
where $\hat{\rho}$ is given for all $x\in[0,1]$ by\footnote{An important observation at this point is that $\hat{p}$ in (\ref{defprnon}) and $\hat{\rho}$ in (\ref{invdefprnon}) are identical. To see this, observe that, through the backstepping transformation, both $\hat{p}$ and $\hat{\rho}$ satisfy the same ODEs in the spatial variable $x$, with the same initial condition $X$. However, $\hat{p}$ is expressed in terms of the original variables $(X,\hat{u})$ in the direct backstepping transformation, whereas $\hat{\rho}$, is expressed in terms of the transformed variables $(X,\hat{w})$ and is used in the inverse transformation.}
\begin{eqnarray}
\hat{\rho}(x,t)&=&X(t)+\hat{D}\int_0^x f\biggl(\hat{\rho}(y,t),\nonumber\\
&&\hat{w}(y,t)+\kappa\left(t+\hat{D}y,\hat{\rho}(y,t)\right)\biggr)dy.\label{invdefprnon}
\end{eqnarray}
System (\ref{nonlper}) together with the control law (\ref{connl}), (\ref{predes1}) can be written in the following form
\begin{eqnarray}
\dot{X}(t)&=&f\left(X(t),\kappa\left(t,X(t)\right)+\hat{w}(0,t)+\tilde{u}(0,t)\right)\label{tilw1}\\
\hat{D}\hat{w}_t(x,t)&=&\hat{w}_x(x,t)+r_1(x,t)\tilde{f}(t)\label{tilw2}\\
\hat{w}(1,t)&=&0,\label{tilw33}
\end{eqnarray}
where 
\begin{eqnarray}
\tilde{f}(t)\!\!=\!\!f\left(\hat{\rho}(0,t),\tilde{u}(0,t)+\hat{u}(0,t)\right)-f\left(\hat{\rho}(0,t),\hat{u}(0,t)\right)\!,\label{fti}
\end{eqnarray}
and $r_1$ is defined in Appendix \ref{secap1}. The observer error 
\begin{eqnarray}
\tilde{u}(x,t)=u(x,t)-\hat{u}(x,t),\label{obsernon}
\end{eqnarray}
 satisfies
\begin{eqnarray}
\tilde{u}_t(x,t)&=&\pi(x,t)\tilde{u}_x(x,t)-\left(1-\hat{D}\pi(x,t)\right)r(x,t)\label{obs3non1}\\
\tilde{u}(1,t)&=&0,\label{obs3non}
\end{eqnarray} 
and
\begin{eqnarray}
\tilde{u}_{xt}(x,t)&=&\pi(x,t)\tilde{u}_{xx}(x,t)+\pi_x(x,t)\tilde{u}_{x}(x,t)-\left(1\right.\nonumber\\
&&\left.-\hat{D}\pi(x,t)\right)r_x(x,t)+\hat{D}\pi_x(x,t)r(x,t)\label{tilu1}\\
\tilde{u}_x(1,t)&=&\left(\frac{1}{\pi(1,t)}-\hat{D}\right)r(1,t),\label{tilu2}
\end{eqnarray}
where 
\begin{eqnarray}
\pi(x,t)&=&\frac{1+x\left(\dot{\sigma}(t)-1\right)}{\sigma(t)-t}\label{defpinon},
\end{eqnarray}
the function $\sigma$ is defined in (\ref{invpnon}), and the function $r$ is defined in Appendix \ref{secap1}. Furthermore,
\begin{eqnarray}
\hat{D}\hat{w}_{xt}(x,t)&=&\hat{w}_{xx}(x,t)+r_2(x,t)\tilde{f}(t)\label{trr1non}\\
\hat{w}_x(1,t)&=&-r_1(1,t)\tilde{f}(t),\label{sys2ntr11non}
\end{eqnarray}
and
\begin{eqnarray}
\hat{D}\hat{w}_{xxt}(x,t)&=&\hat{w}_{xxx}(x,t)+r_3(x,t)\tilde{f}(t)\label{trr1non1}\\
\hat{w}_{xx}(1,t)&=&\!-r_2(1,t)\tilde{f}(t)+r_4(t)\tilde{f}(t)+\tilde{f}^T(t)r_5(t)\tilde{f}(t)\nonumber\\
&&\!-r_1(1,t)\tilde{f}_{\hat{\rho}}(t)f(\hat{\rho}(0,t),\tilde{u}(0,t)+\hat{u}(0,t))\nonumber\\
&&\!-r_1(1,t)\hat{D}r(0,t)\tilde{f}_{\hat{u}}+\hat{D}r_1(1,t)r(0,t)\left(1\right.\nonumber\\
&&\!\left.-\hat{D}\pi(0,t)\right)\frac{\partial f\left(\hat{\rho}(0,t),\tilde{u}(0,t)+\hat{u}(0,t)\right)}{\partial \hat{u}}\nonumber\\
&&\!-\hat{D}r_1(1,t)\pi(0,t)\tilde{u}_x(0,t)\nonumber\\
&&\!\times\frac{\partial f\left(\hat{\rho}(0,t),\tilde{u}(0,t)+\hat{u}(0,t)\right)}{\partial \hat{u}}\label{sys2ntr11nonxx},
\end{eqnarray}
where
\begin{eqnarray}
\tilde{f}_{\hat{\rho}}(t)&=&\frac{\partial f\left(\hat{\rho}(0,t),\tilde{u}(0,t)+\hat{u}(0,t)\right)}{\partial \hat{\rho}}\nonumber\\
&&-\frac{\partial f\left(\hat{\rho}(0,t),\hat{u}(0,t)\right)}{\partial \hat{\rho}}\label{fzz}\\
\tilde{f}_{\hat{u}}(t)&=&\frac{\partial f\left(\hat{\rho}(0,t),\tilde{u}(0,t)+\hat{u}(0,t)\right)}{\partial \hat{u}}\nonumber\\
&&-\frac{\partial f\left(\hat{\rho}(0,t),\hat{u}(0,t)\right)}{\partial \hat{u}},\label{fuu}
\end{eqnarray}
and $r_2$, $r_3$, $r_4$ and $r_5$ are defined in Appendix \ref{secap1}.
\end{lemma}

\begin{proof}
We re-write system (\ref{nonlper}) in the form
\begin{eqnarray}
\dot{X}(t)&=&f\left(X(t),u(0,t)\right)\label{sys2non}\\
u_t(x,t)&=&\pi(x,t)u_x(x,t)\label{gfgf}\\
u(1,t)&=&U(t)\label{sys2nnon},
\end{eqnarray}
where the actual actuator state $u(x,t)$ (for which an observer is given in (\ref{obs1non})--(\ref{obs2non})) satisfies (\ref{gfgf})--(\ref{sys2nnon}) and
\begin{eqnarray}
u(x,t)=U\left(\phi\left(t+x\left(\sigma(t)-t\right)\right)\right),\quad \mbox{for all $x\in[0,1]$},\label{defrunon}
\end{eqnarray}
where $\phi$ is defined in  (\ref{defpnon}), or, incorporating $\delta$, as
\begin{eqnarray}
u(x,t)&=&U\biggl(t+x\left(\hat{D}+\delta\left(\sigma(t),X(\sigma(t))\right)\right)-\hat{D}\nonumber\\
&&-\delta\left(t+x\left(\hat{D}+\delta\left(\sigma(t),X(\sigma(t))\right)\right)\right)\biggr).\label{defrudeltanon}
\end{eqnarray}
The rest of the lemma is proved with lengthy but straightforward calculations and using Lemma \ref{lemdp} from Appendix \ref{appc}, the backstepping transformation (\ref{back1non}), its inverse (\ref{invback1non}), and (\ref{defprnon}), (\ref{obs1non})--(\ref{obs2non}), (\ref{sys2non})--(\ref{sys2nnon}), (\ref{invdefprnon}), (\ref{obsernon}). 
\end{proof}

For $\pi$ to be a meaningful propagation speed it should be positive and uniformly bounded from below and above. Using (\ref{defpinon}), one can conclude that $\delta$, the initial conditions and the solutions of the system should satisfy, in addition to (\ref{prhx}) which guarantees that $0<\dot{\sigma}(t)$, also
\begin{eqnarray}
 0<\hat{D}+\delta(\sigma(\theta),P^*(\theta)),\quad\!\!\! \mbox{for all $\theta\geq t_0-D\left(X(t_0)\right)$}\label{prhxnew},
 \end{eqnarray}
which guarantees $0<\sigma(t)-t$. The two conditions (\ref{prhx}), (\ref{prhxnew}) incorporate the functions $\sigma$ and $P^*$, that is, they are not expressed in terms of the perturbation $\delta$ and the functional $\Pi$. We derive next a sufficient condition for (\ref{prhx}), (\ref{prhxnew}) to be satisfied, in terms of $\Pi$.

\begin{lemma}
\label{lemmanm}
There exist positive constants $c_1$, $c^*$, such that if the perturbation $\delta$ satisfies (\ref{89}), then for all solutions of the system satisfying, 
\begin{eqnarray}
\Pi(t)<c^*,\label{222}
\end{eqnarray}
 they also satisfy 
 \begin{eqnarray}
 \left(c_1+\hat{\mu}\left(\left|P^*(\theta)\right|\right)\right)\left(1+\left|f\left(P^*(\theta),U(\theta)\right)\right|\right)&<&R\label{136},
 \end{eqnarray}
for all $\phi(t)\leq\theta\leq t$, where
\begin{eqnarray}
R=\min\left\{r^*,c,\hat{D}\right\},
\end{eqnarray}
 for some $0<{c}<1$, and hence, conditions (\ref{prhx}) for $0<c<1$, and (\ref{prhxnew}) are satisfied.
\end{lemma}

\begin{proof}
See Appendix \ref{secprb}.
\end{proof}


\begin{lemma}
\label{lyapl}
There exist a continuously differentiable, positive definite function $S^*$, a class $\mathcal{K}_{\infty}$ function $\alpha^*$ and positive constants $\lambda$, $c_1$, $c^*$ such that if the perturbation $\delta$ satisfies (\ref{89}) then for all solutions of the system satisfying (\ref{222}), the Lyapunov function
\begin{eqnarray}
V(t)&=&S^*\left(X(t)\right)+g_{11}\int_0^1e^{g_1x}|\tilde{u}(x,t)|dx+g_6\int_0^1e^{g_2x}\nonumber\\
&&\times\tilde{u}_x(x,t)^2dx+g_9\hat{D}\int_0^1e^{g_{10}x}\left|\hat{w}_x(x,t)\right|dx\nonumber\\
&&+g_{12}\hat{D}\int_0^1e^{g_3x}\alpha^*(|\hat{w}(x,t)|)dx+g_8\hat{D}\int_0^1e^{g_5x}\nonumber\\
&&\times\hat{w}_{xx}(x,t)^2dx+g_7\hat{D}\int_0^1e^{g_4x}\hat{w}_x(x,t)^2dx,
\end{eqnarray}
where 
$g_i>0$, $i=1\ldots 12$, satisfies 
\begin{eqnarray}
V(t)\leq V(t_0)e^{-\lambda(t-t_0)}, \quad \mbox{for all $t\geq t_0$}.\label{Afth}
\end{eqnarray}
\end{lemma}

\begin{proof}
See Appendix \ref{apply}.
\end{proof}

The next two lemmas relate the Lyapunov function $V$ with the norm of the system in the original variables, represented with PDEs, and the norm in PDE representation with the norm in standard delay form.

\begin{lemma}
\label{lemag}
There exist a positive constant $c^*$, a class $\mathcal{KC}_{\infty}$ function $\alpha_{24}$ and a class $\mathcal{K}_{\infty}$ function $\alpha_{25}$ such that for all solutions of the system satisfying (\ref{222}) the following holds
\begin{eqnarray}
\alpha_{24}\left(\Gamma(t),R\right)\leq V(t)\leq \alpha_{25}\left(\Gamma(t)\right),
\end{eqnarray}
where
\begin{eqnarray}
\Gamma(t)&=&|X(t)|+\int_0^1\alpha^*\left(\left|\hat{u}(x,t)\right|\right)dx+\int_0^1{u}_x(x,t)^2dx\nonumber\\
&&+\int_0^1\hat{u}_x(x,t)^2dx+\int_0^1\hat{u}_{xx}(x,t)^2dx.\label{ghh}
\end{eqnarray}
\end{lemma}

\begin{proof}
See Appendix \ref{appp}.
\end{proof}

\begin{lemma}
\label{lemga}
There exists positive constants $c^*$, $c_1$ and class $\mathcal{KC}_{\infty}$ functions $\zeta_{1}$ and $\zeta_{2}$ such that if the perturbation $\delta$ satisfies (\ref{89}), then for all solutions of the system satisfying (\ref{222}) the following holds 
\begin{eqnarray}
\zeta_{1}\left(\Pi(t),R\right)\leq \Gamma(t)\leq \zeta_{2}\left(\Pi(t),R\right).
\end{eqnarray}
\end{lemma}

\begin{proof}
Using Lemma \ref{lemmanm} and (\ref{invpnon}), (\ref{siinon}) we get that
\begin{eqnarray}
\frac{1}{\hat{D}+R}&\leq&\frac{1}{\sigma(t)-t}\leq\frac{1}{\hat{D}-R}\label{put1}\\
\frac{1}{1+R}&\leq&\dot{\sigma}(\theta)\leq\frac{1}{1-R},\quad \mbox{for all $\phi(t)\leq\theta\leq t$}\label{put}.
\end{eqnarray}
With relations (\ref{est11}), (\ref{defrunon}), (\ref{defgadelnon}), (\ref{ghh}) and applying the appropriate change of variables in the integrals the proof is immediate using (\ref{put1}), (\ref{put}).
\end{proof}

\begin*{{\it Proof of Theorem \ref{thmpernon}}:}
Using Lemma \ref{lemga} we conclude that (\ref{222}) is satisfied if $\Gamma(t)\leq \zeta_{1}\left(c^*,R\right)$, and hence, with Lemma \ref{lemag}, (\ref{222}) is satisfied if 
\begin{eqnarray}
V(t)\leq\alpha_{24}\left( \zeta_{1}\left(c^*,R\right),R\right),\label{bigas}
\end{eqnarray}
 is satisfied. Assume for the moment that (\ref{bigas}) is satisfied. With Lemmas \ref{lyapl}, \ref{lemag} and \ref{lemga}, relation (\ref{bigas}) is satisfied if $c^{**}$ in (\ref{normth}) is such that
\begin{eqnarray}
c^{**}\leq \bar{\zeta}_{2}\left(\alpha_{25}^{-1}\left(\alpha_{24}\left( \zeta_{1}\left(c^*,R\right),R\right)\right),R\right).
\end{eqnarray}
Using (\ref{Afth}), with some routine class $\mathcal{K}$ majorizations that involve Lemmas \ref{lemag} and \ref{lemga}, we get estimate (\ref{bouthm}).
\end*\quad

\section{Robustness to time-varying delay perturbation for linear systems}
\label{sectionlinear}
We consider the following special case of system (\ref{nonlper})
\begin{eqnarray}
\dot{X}(t)=AX(t)+BU\left(t-\hat{D}-\delta(t)\right),\label{sys1}
\end{eqnarray}
where for the rest of the section $\delta$ is a function only of the time $t$. For this linear case, the predictor-based controller is given explicitly as
\begin{eqnarray}
U(t)=Ke^{A\hat{D}}X(t)+\int_{t-\hat{D}}^{t}e^{A(t-\theta)}BU(\theta)d\theta.\label{con1m}
\end{eqnarray}

\begin{theorem}
\label{thmper}
Consider the closed-loop system consisting of the plant (\ref{sys1}), observer (\ref{obs1non})--(\ref{obs2non}), and control law (\ref{con1m}). There exists a positive $\delta_1$, such that if the perturbation $\delta$ satisfies
\begin{eqnarray}
|\delta(t)|+|\delta'(t)|&<&\delta_1,\quad \mbox{for all $t\geq t_0$}\label{case1},
\end{eqnarray}
then, the closed-loop system is exponentially stable, in the sense that there exist positive constants $R$ and $\lambda$ such that the following holds:
\begin{eqnarray}
\Pi_{\rm L}(t)&\leq& R\Pi_{\rm L}(t_0)e^{-\lambda (t-t_0)}\label{estt}\\
\Pi_{\rm L}(t)&=&\left|X(t)\right|^2+\int_{t-\hat{D}-\max\left\{0,\delta(t)\right\}}^tU(\theta)^2d\theta\nonumber\\
&&+\int_{t-\hat{D}}^t\dot{U}(\theta)^2d\theta.\label{defgadel}
\end{eqnarray}
\end{theorem}

The proof of Theorem \ref{thmper} is based on the application of Lemmas \ref{lemma1non}, \ref{lyapl}, \ref{lemag} and \ref{lemga} for the special case of plant (\ref{sys1}) and the special perturbation $\delta$ as in (\ref{sys1}). However, we give each of these lemmas specialized to the present case for two reasons. Firstly, in the special case of linear systems with only time-varying perturbation, we study stability of the closed-loop system in the $H_1$ norm. This is a consequence of the fact that when $\delta$ does not depend on the state, the conditions (\ref{prhx}), (\ref{prhxnew}) are satisfied without restricting the supremum norm of the real actuator state $U(\theta)$, for all $\phi(t)\leq\theta\leq t$. Secondly, in the linear case the control law, as well as the direct and inverse backstepping transformations are given explicitly and are globally well-defined. 

When (\ref{case1}) does not hold, one can still derive exponential stability of the closed-loop system, by imposing other conditions on the perturbation $\delta$, such us the ones from \cite{khalil} Chapter 9.3. However, in this case one has to guarantee in addition that the propagation speed $\pi$ is still uniformly bounded from above and below and strictly positive. Therefore, we make the following assumptions which $\delta$ has to a priori satisfy. 

\begin{assumption} 
\label{as1}
The perturbation $\delta$ satisfies
\begin{eqnarray}
\delta'(t)<1,\quad \mbox{for all $t\geq t_0$},
\end{eqnarray}
and is such that
\begin{eqnarray}
\pi_1^*&=&\frac{1}{\sup_{\theta\geq\sigma(t_0)}(1-\delta'(\theta))}>0\label{ass1bnew}.
\end{eqnarray}
\end{assumption}

\begin{assumption} 
\label{as2}
The perturbation $\delta$ satisfies
\begin{eqnarray}
\hat{D}+\delta(t)>0,\quad \mbox{for all $t\geq t_0$},
\end{eqnarray}
and is such that
\begin{eqnarray}
\pi_0^*&=&\frac{1}{\sup_{\theta\geq\sigma(t_0)}(\hat{D}+\delta(\theta))}>0.
\end{eqnarray}
\end{assumption}

\begin{theorem}
\label{theorem3}
Assume that $\delta$ satisfy Assumptions \ref{as1}, \ref{as2}. There exist positive $\delta_2$ and $\delta_3$ such that if the perturbation $\delta$ satisfies either
\begin{eqnarray}
\int_0^{\infty}\left(\left|\delta'(\theta)\right|+\left|\delta(\theta)\right|\right)d\theta&\leq& \delta_2\label{case2},
\end{eqnarray}
or
\begin{eqnarray}
|\delta(t)|+|\delta'(t)|&\to&0,\quad \mbox{when $t\to\infty$} \label{case3},
\end{eqnarray}
or
\begin{eqnarray}
\frac{1}{\Delta}\int_t^{t+\Delta}\left(\left|\delta'(\theta)\right|+\left|\delta(\theta)\right|\right)d\theta&\leq& \delta_3\quad \mbox{for all $t\geq T$},\label{case4}
\end{eqnarray}
for some positive $\Delta$ and nonnegative $T$, then, the closed-loop system consisting of the plant (\ref{sys1}), observer (\ref{obs1non})--(\ref{obs2non}), and control law (\ref{con1m}) is exponentially stable, in the sense that there exist positive constants $R$ and $\lambda$ such that the following holds:
\begin{eqnarray}
\Pi_{\rm L}(t)&\leq& R\Pi_{\rm L}(t_0)e^{-\lambda (t-t_0)}\\
\Pi_{\rm L}(t)&=&\left|X(t)\right|^2+\int_{t-\hat{D}-\max\left\{0,\delta(t)\right\}}^tU(\theta)^2d\theta\nonumber\\
&&+\int_{t-\hat{D}}^t\dot{U}(\theta)^2d\theta.\label{defgadel2case}
\end{eqnarray}
\end{theorem}

From Theorem \ref{theorem3} we observe that in the case of linear systems and when the perturbation depends only on time and not on the state, we prove robustness of global exponential stability of the predictor feedback under three alternative conditions on the delay perturbation rather than just restricting the magnitude of the delay and its rate. This is not possible in the case where the perturbation depends on the state because for the conditions (\ref{prhx}) and (\ref{prhxnew}) to be satisfied one has to necessarily restrict the magnitude of the perturbation and its rate such that they are both sufficiently small.

We introduce now the backstepping transformation of the estimated actuator state.

\begin{lemma}
\label{lemena}
Consider the backstepping transformation 
\begin{eqnarray}
\hat{w}(x,t)&=&\hat{u}(x,t)-Ke^{A\hat{D}x}X(t)\nonumber\\
&&-\hat{D}K\int_0^xe^{A\hat{D}(x-y)}B\hat{u}(y,t)dy,\label{bck1}
\end{eqnarray}
together with its inverse given by
\begin{eqnarray}
\hat{u}(x,t)&=&\hat{w}(x,t)+Ke^{(A+BK)\hat{D}x}X(t)\nonumber\\
&&+\hat{D}K\int_0^xe^{(A+BK)\hat{D}(x-y)}B\hat{w}(y,t)dy.\label{bckin}
\end{eqnarray}
System (\ref{sys1}) together with the control law (\ref{con1m}) can be represented as
\begin{eqnarray}
\dot{X}(t)&=&\left(A+BK\right)X(t)+B\hat{w}(0,t)+B\tilde{u}(0,t)\label{sys2tr}\\
\hat{D}\hat{w}_t(x,t)&=&\hat{w}_x(x,t)-\hat{D}Ke^{A\hat{D}x}B\tilde{u}(0,t)\label{sys2new}\\
\hat{w}(1,t)&=&0,\label{sys2ntr}
\end{eqnarray}
where the observer error 
\begin{eqnarray}
\tilde{u}(x,t)=u(x,t)-\hat{u}(x,t),\label{obserr}
\end{eqnarray}
satisfies
\begin{eqnarray}
\tilde{u}_t(x,t)&=&\pi(x,t)\tilde{u}_x(x,t)-\left({1}-\hat{D}\pi(x,t)\right) r(x,t)\label{til1}\\
\tilde{u}(1,t)&=&0,\label{til2}
\end{eqnarray}
with
\begin{eqnarray}
\pi(x,t)&=&\frac{1+x\left(\dot{\sigma}(t)-1\right)}{\sigma(t)-t}\label{defpi}
\end{eqnarray}
\begin{eqnarray}
\phi(t)&=&t-\hat{D}-\delta(t)\label{defp}\\
\sigma(t)&=&\phi^{-1}(t)\nonumber\\
&=&t+\hat{D}+\delta(\sigma(t))\label{invp}\\
r(x,t)&=& \frac{1}{\hat{D}}  \hat{w}_x(x,t)+Ke^{(A+BK)\hat{D}x}(A+BK)X(t)\nonumber\\
&&+KB\hat{w}(x,t)+\hat{D}K(A+BK)\nonumber\\
&&\times\int_0^xe^{(A+BK)\hat{D}(x-y)}B\hat{w}(y,t)dy.\label{roo}
\end{eqnarray}
Furthermore,
\begin{eqnarray}
\hat{D}\hat{w}_{xt}(x,t)&=&\hat{w}_{xx}(x,t)-\hat{D}^2Ke^{A\hat{D}x}AB\tilde{u}(0,t)\label{trr1}\\
\hat{w}_x(1,t)&=&\hat{D}Ke^{A\hat{D}}B\tilde{u}(0,t).\label{sys2ntr11}
\end{eqnarray}
\end{lemma}

\begin{proof}
System (\ref{sys1}) can be re-written in the form
\begin{eqnarray}
\dot{X}(t)&=&AX(t)+Bu(0,t)\label{sys2}\\
u_t(x,t)&=&\pi(x,t)u_x(x,t)\\
u(1,t)&=&U(t),\label{sys2n}
\end{eqnarray}
where $\pi$ is defined in (\ref{defpi}). The predictor feedback (\ref{con1m}) can be written as
\begin{eqnarray}
U(t)=Ke^{A\hat{D}}X(t)+\hat{D}\int_0^{1}e^{A\hat{D}(1-y)}B\hat{u}(y,t)dy,
\end{eqnarray}
where the estimation of the unmeasured actuator state $U(\theta)$, for all $t-\hat{D}-\delta(t)\leq\theta\leq t$, $\hat{u}(x,t)$ is defined in (\ref{obs1non}), (\ref{obs2non}) and satisfies (\ref{est11}). With representation (\ref{sys2})--(\ref{sys2n}) for system (\ref{sys1}), the actuator state $u(x,t)$ is
\begin{eqnarray}
u(x,t)=U\left(\phi\left(t+x\left(\sigma(t)-t\right)\right)\right),\quad \mbox{for all $x\in[0,1]$},\label{defru}
\end{eqnarray}
or, incorporating $\delta$, $u(x,t)=U\left(t+x\left(\hat{D}+\delta(\sigma(t))\right)\right.$ $\left.-\hat{D}-\delta\left(t+x\left(\hat{D}+\delta(\sigma(t))\right)\right)\right)$ for all $x\in[0,1]$. Since the perturbation $\delta$ satisfies (\ref{case1}), it follows from definition (\ref{defpi}) and relations (\ref{defp}), (\ref{invp}) that $\sigma(t)-t>0$ and that ${1-\delta'(t)}>0$, for all $t\geq t_0$. Define the quantities 
\begin{eqnarray}
\pi_1^{**}&=&\frac{1}{\sup_{\theta\geq\sigma(t_0)}(1-\delta'(\theta))}\label{ass1b}\\
\pi_0^{**}&=&\frac{1}{\sup_{\theta\geq\sigma(t_0)}(\hat{D}+\delta(\theta))}.
\end{eqnarray}
From (\ref{case1}) it follows that $\sup_{\theta\geq\sigma(0)}(1-\delta'(\theta))<\infty$ and that $\sup_{\theta\geq\sigma(0)}(\hat{D}+\delta(\theta))<\infty$, and hence, $\pi_1^{**}>0$, $\pi_0^{**}>0$. Since $\sigma(t)-t=\hat{D}+\delta(\sigma(t))$ and $\dot{\sigma}(t)=\frac{1}{1-\delta'(\sigma(t))}$, using (\ref{case1}) we conclude that $\pi$ is positive and uniformly bounded from above and below. Hence, $\pi$ is a meaningful propagation speed. The rest of the proof is based on algebraic manipulations and it is omitted. 
\end{proof}



\begin{lemma}
\label{lem3}
There exist positive constants $r_1$ and $r_2$ such that the derivative of the Lyapunov function 
\begin{eqnarray}
V_{\rm L}(t)&=&X(t)^TPX(t)+b_1\int_0^1e^{bx}\tilde{u}(x,t)^2dx\nonumber\\
&&+\hat{D}b_2\int_0^1(1+x)\hat{w}(x,t)^2dx\nonumber\\
&&+\hat{D}b_2\int_0^1(1+x)\hat{w}_x(x,t)^2dx,\label{vfac}
\end{eqnarray}
along the solutions of (\ref{sys2tr})--(\ref{sys2ntr}), (\ref{til1})--(\ref{til2}), (\ref{trr1})--(\ref{sys2ntr11}) satisfies
\begin{eqnarray}
\dot{V}_{\rm L}(t)&\leq& -r_1V_{\rm L}(t)+r_2\gamma(t)V_{\rm L}(t),\label{vdottt}\\
\gamma(t)&=&\max\left\{\left|\delta\left(\sigma(t)\right)\right|,\right.\nonumber\\
&&\left.\left|\delta(\sigma(t))\left(1-\delta'(\sigma(t))\right)-\hat{D}\delta'(\sigma(t))\right|\right\}\label{starn}.
\end{eqnarray}
\end{lemma}

\begin{proof}
See Appendix \ref{lin1ap}.
\end{proof}

\begin{lemma}
\label{lemfour}
There exists a positive $\delta_1$ such that if the perturbation $\delta$ satisfies (\ref{case1}), then there exists a positive $\lambda$ such that $V$ in (\ref{vfac}) satisfies
\begin{eqnarray}
\dot{V}_{\rm L}(t)\leq -\lambda V_{\rm L}(t).
\end{eqnarray}
\end{lemma}

\begin{proof}
See Appendix \ref{fac1}
\end{proof}

\begin{lemma}
\label{lem4}
There exist positive constants $M_{5,\rm L}$ and $M_{6,\rm L}$ such that
\begin{eqnarray}
M_{5,\rm L}V_{\rm L}(t)&\leq&\Gamma_{\rm L}(t)\leq M_{6,\rm L}V_{\rm L}(t)\\
\Gamma_{\rm L}(t)&=&\left|X(t)\right|^2+\int_0^1u(x,t)^2dx+\int_0^1\hat{u}(x,t)^2dx\nonumber\\
&&+\int_0^1\hat{u}_x(x,t)^2dx.\label{defga}
\end{eqnarray}
\end{lemma}

\begin{proof}
See Appendix \ref{fac2}.
\end{proof}

\begin{lemma}
\label{lem5}
There exist positive constants $M_{7,\rm L}$ and $M_{8,\rm L}$ such that
\begin{eqnarray}
M_{7,\rm L}\Gamma_{\rm L}(t)\leq\Pi_{\rm L}(t)\leq M_{8,\rm L}\Gamma_{\rm L}(t).
\end{eqnarray}
\end{lemma}

\begin{proof}
From (\ref{est11}) we get $\hat{u}_x(x,t)\!=\!\hat{D}U'\!\left(t\!+\!\hat{D}(x-1)\right)$. Applying a change of variables in (\ref{defga}), with (\ref{defru}) we get
\begin{eqnarray}
\Gamma_{\rm L}(t)&=&\left|X(t)\right|^2+\frac{1}{\sigma(t)-t}\int_{t-\hat{D}-\delta(t)}^{t}\frac{1}{\phi'\left(\sigma(\theta)\right)}U(\theta)^2d\theta\nonumber\\
&&+\frac{1}{\hat{D}}\int_{t-\hat{D}}^tU(\theta)^2d\theta+\int_{t-\hat{D}}^t\dot{U}(\theta)^2d\theta.
\end{eqnarray}
Hence, the lemma is proved with $M_{7,\rm L}=\frac{1}{\max\left\{1,M_{\rm L}+\frac{1}{\hat{D}}\right\}}$, $M_{\rm L}=\frac{\sup_{t-\hat{D}-\delta(t)\leq\theta\leq t}\frac{1}{1-\delta'(\sigma(\theta))}}{\inf_{t\geq t_0}\left(\hat{D}+\delta(\sigma(t))\right)}$, $M_{8,\rm L}=\frac{1}{\min\left\{1,\pi_0^{**}\pi_1^{**},\frac{1}{\hat{D}}\right\}}$.
\end{proof}

\begin*{{\it Proof of Theorem \ref{thmper}}:}
Using Lemma \ref{lemfour} and the comparison principle (\cite{khalil}), we get $V_{\rm L}(t)\leq V_{\rm L}(t_0)e^{-\lambda (t-t_0)}$. With Lemmas \ref{lem4} and \ref{lem5} we get (\ref{estt}) with $R=\frac{M_{6,\rm L}M_{8,\rm L}}{M_{5,\rm L}M_{7,\rm L}}$.
\end*\qed

\begin*{{\it Proof of Theorem \ref{theorem3}}:}
We consider first the case where $\delta$ satisfies (\ref{case2}). Under Assumptions \ref{as1} and \ref{as2}, Lemmas \ref{lemena}, \ref{lem3}, \ref{lem4} and \ref{lem5} apply to this case as well (with $\pi_0^{**}$ and $\pi_1^{**}$ replaced by $\pi_0^{*}$ and $\pi_1^{*}$ respectively). The only difference with the proof of Theorem \ref{thmper} is in the proof of Lemma \ref{lemfour}. Towards that end, we solve (\ref{vdottt}) to get
\begin{eqnarray}
V_{\rm L}(t)\leq e^{-r_1(t-t_0)+r_2\int_{t_0}^t\gamma(\tau)d\tau}V_{\rm L}(t_0).
\end{eqnarray}
Consider that $\gamma(t)=\left|\frac{1}{\pi(0,t)}-\hat{D}\right|$. Applying the change of variables $\tau=\phi(\theta)$ in the integral and using the facts that $\sigma(t_0)>t_0$ and $\phi'(t)=1-\delta'(t)$ we get
\begin{eqnarray}
\int_{t_0}^{\infty}\gamma(\tau)d\tau\leq \frac{1}{\pi_1^*} \int_{t_0}^{\infty}\left|\delta\left(\theta\right)\right|d\theta\leq\frac{\delta_2}{\pi_1^*}.
\end{eqnarray}
Analogously, for the case $\gamma(t)=\left|\frac{1}{\pi(1,t)}-\hat{D}\right|$, we get
\begin{eqnarray}
\int_{t_0}^{\infty}\gamma(\tau)d\tau\leq \frac{\delta_2}{\pi_1^*}\left(\hat{D}+\frac{1}{\pi_1^*}\right).
\end{eqnarray}
Hence, Lemma \ref{lemfour} is proved with $r_2\delta_2<\min\left\{{\pi_1^*r_1}\right.,$ $\left.\frac{{\pi_1^*}^2r_1}{\left(1+{\pi_1^*}\hat{D}\right)}\right\}$ and the fact that $\int_{t_0}^t\gamma(\tau)d\tau\leq\int_{t_0}^{\infty}\gamma(\tau)d\tau$. Note that in the present case, Lemma \ref{lemfour} can be proved directly from relation (\ref{vdottt}) using Lemma B.6 in \cite{KKK}.

For the case where $\delta$ satisfies (\ref{case3}), Lemma \ref{lemfour} is proved by combining Lemma \ref{lem3}, Lemma B.8 in \cite{KKK} and the fact that $\gamma(t)$ satisfies
\begin{eqnarray}
\gamma(t)\leq \left(1+\frac{1}{\pi_1^*}+\hat{D}\right)\left( |\delta(\sigma(t))|+|\delta'(\sigma(t))|\right).
\end{eqnarray}  
 Finally, if $\delta$ satisfies (\ref{case4}), Lemma \ref{lemfour} is proved using Lemma 9.5 in \cite{khalil}.
\end*\qed

\section{Examples}
\label{secex}

\subsection{Control of a DC motor over a network}
We consider the following model of a field-controlled DC motor (\cite{slemon}) with  negligible shaft damping 
\begin{eqnarray}
\frac{d\omega(t)}{dt}&=&\theta i_{\rm f}(t)i_{\rm a}(t)\label{nomm1}\\
\frac{d{i}_{\rm a}(t)}{dt}&=&\!-bi_{\rm a}(t)\!+\!k\!-\!ci_{\rm f}(t)\omega(t)\\
\frac{d{i}_{\rm f}(t)}{dt}&=&\!-ai_{\rm f}(t)\!+\!U\left(t\!-\!\hat{D}\!-\!\rho(t,i_{\rm f}(t),i_{\rm a}(t),\omega(t))\right)\!,\label{nommn}
\end{eqnarray}
where $i_{\rm f}$, $i_{\rm a}$ are field and armature currents respectively, $\omega$ is angular velocity and $a$, $b$, $c$, $\theta$ are positive constants. The equilibria of the unforced system are $(\omega,i_{\rm a},i_{\rm f})=\left(\omega_0,\frac{k}{b},0\right)$ for some constant $\omega_0$. The system is feedback linearizable for $(\omega,i_{\rm a},i_{\rm f})\in D, D=$ $\left\{(\omega,i_{\rm a},i_{\rm f})\in\mathbb{R}^3|\mbox{$\omega>0$ and $i_{\rm a}\!>\!\frac{k}{2b}$} \right\}$. A delay-free design, based on full-state linearization, is (Chapter 13.3 in \cite{khalil})
\begin{eqnarray}
U(t)&=&\frac{1}{\gamma}\left(-K_1Z_1(t)-K_2Z_2(t)-K_3Z_3(t)-\alpha\right)\label{ex1c}\\
Z_1(t)&=&\theta i_{\rm a}(t)^2+c\omega(t)^2-\theta\frac{k^2}{b^2}-c\omega_0^2\\
Z_2(t)&=&2\theta i_{\rm a}(t)\left(k-bi_{\rm a}(t)\right)\\
Z_3(t)&=&2\theta\left(k-2bi_{\rm a}(t)\right)\left(-bi_{\rm a}(t)+k-ci_{\rm f}(t)\omega(t)\right)\\
\gamma&=&-2c\theta\left(k-2bi_{\rm a}(t)\right)\omega(t)\\
\alpha&=&2ca\theta\left(k-2bi_{\rm a}(t)\right)i_{\rm f}(t)\omega(t)-2b\theta\left(3k-4bi_{\rm a}(t)\right.\nonumber\\
&&\left.-2ci_{\rm f}(t)\omega(t)  \right)\left(-bi_{\rm a}(t)+k-ci_{\rm f}(t)\omega(t)\right)\nonumber\\
&&-2c\theta\left(k-2bi_{\rm a}(t)\right)i_{\rm f}(t)^2\omega(t).\label{exnc}
\end{eqnarray}  

Shifting the equilibrium $\left(\omega_0,\frac{k}{b},0\right)$ of the system to the origin and setting $X_1=\omega-\omega_0$, $X_2=i_{\rm a}-\frac{k}{b}$, $X_3=i_{\rm f}$, $\delta(t,X_1(t),X_2(t),X_3(t))=\rho(t,i_{\rm f}(t),i_{\rm a}(t),\omega(t))$ we get
\begin{eqnarray}
\dot{X}_1(t)&=&\theta X_2(t)X_3(t)+\frac{\theta k}{b} X_3(t)\\
\dot{X}_2(t)&=&-bX_2(t)-cX_3(t)X_1(t)-c\omega_0X_3(t)\\
\dot{X}_3(t)&=&-aX_3(t)\nonumber\\
&&+U\left(t-\hat{D}-\delta(t,X_1(t),X_2(t),X_3(t))\right).
\end{eqnarray}
The motor is controlled through a network that induces a constant delay $\hat{D}$ (e.g. \cite{dcf}). The known, constant delay, is subject to a time-varying perturbation due to the effect of transmission of control signals to other motors through the network. We further assume that the perturbation $\delta$ increases when the armature current increases. Define the estimated predictors of $X_1$, $X_2$ and $X_3$ as
\begin{eqnarray}
\hat{P}_1(t)&=&X_1(t)+\theta\int_{t-\hat{D}}^t\left(\hat{P}_2(s)\hat{P}_3(s)+\frac{k}{b}\hat{P}_3(s)\right)ds\label{prex1}\\
\hat{P}_2(t)&=&X_2(t)+\int_{t-\hat{D}}^t\biggl(-b\hat{P}_2(s)-c\hat{P}_1(s)\hat{P}_3(s)\nonumber\\
&&-c\omega_0\hat{P}_3(s)\biggr)ds\\
\hat{P}_3(t)&=&X_3(t)+\int_{t-\hat{D}}^t\left(-a\hat{P}_3(s)+U(s)\right)ds,\label{dum}
\end{eqnarray}
respectively. Setting in (\ref{ex1c})--(\ref{exnc}) $\omega=X_1+\omega_0$, $i_{\rm a}=X_2+\frac{k}{b}$, $i_{\rm f}=X_3$ and replacing $X_1$, $X_2$, $X_3$ by the predictors (\ref{prex1})--(\ref{dum}) we get the nominal predictor feedback.

We choose the set-point for the angular velocity of the motor as $\omega_0=2$, the nominal delay $\hat{D}=1$ and the parameters of the plant as $a=b=c=k=\theta=1$. The delay perturbation is $\delta\left(t,X_2(t)\right)=0.5\left(X_2(t)+\frac{k}{b}\right)^2+0.2\sin(t)^2$. The initial conditions for the plant and the actuator state are chosen as $X_1(0)=-1$, $X_2(0)=-0.2$, $X_3(0)=0.1$ and $U(\theta)=0$, $-1-0.5\left(X_2(0)+1\right)^2\leq\theta\leq 0$ respectively. The parameters of the controller are chosen as $K_1=-1$, $K_2=K_3=-3$, such as the linearizable, delay-free, system (i.e., the delay-free plant in the $Z$ coordinates) has three eigenvalues at $-1$, and the initial estimate of the actuator state as $U(\theta)=0$, $-1\leq\theta\leq0$.

In Fig. \ref{figi} we show the field and armature currents, and in Fig. \ref{figomega} the input voltage and the angular velocity of the motor. The nominal predictor feedback achieves local stabilization of the closed-loop system at the desired equilibrium, despite the presence of the perturbation. 

\begin{figure}[t]
\centering
		\includegraphics[width=.42\textwidth]{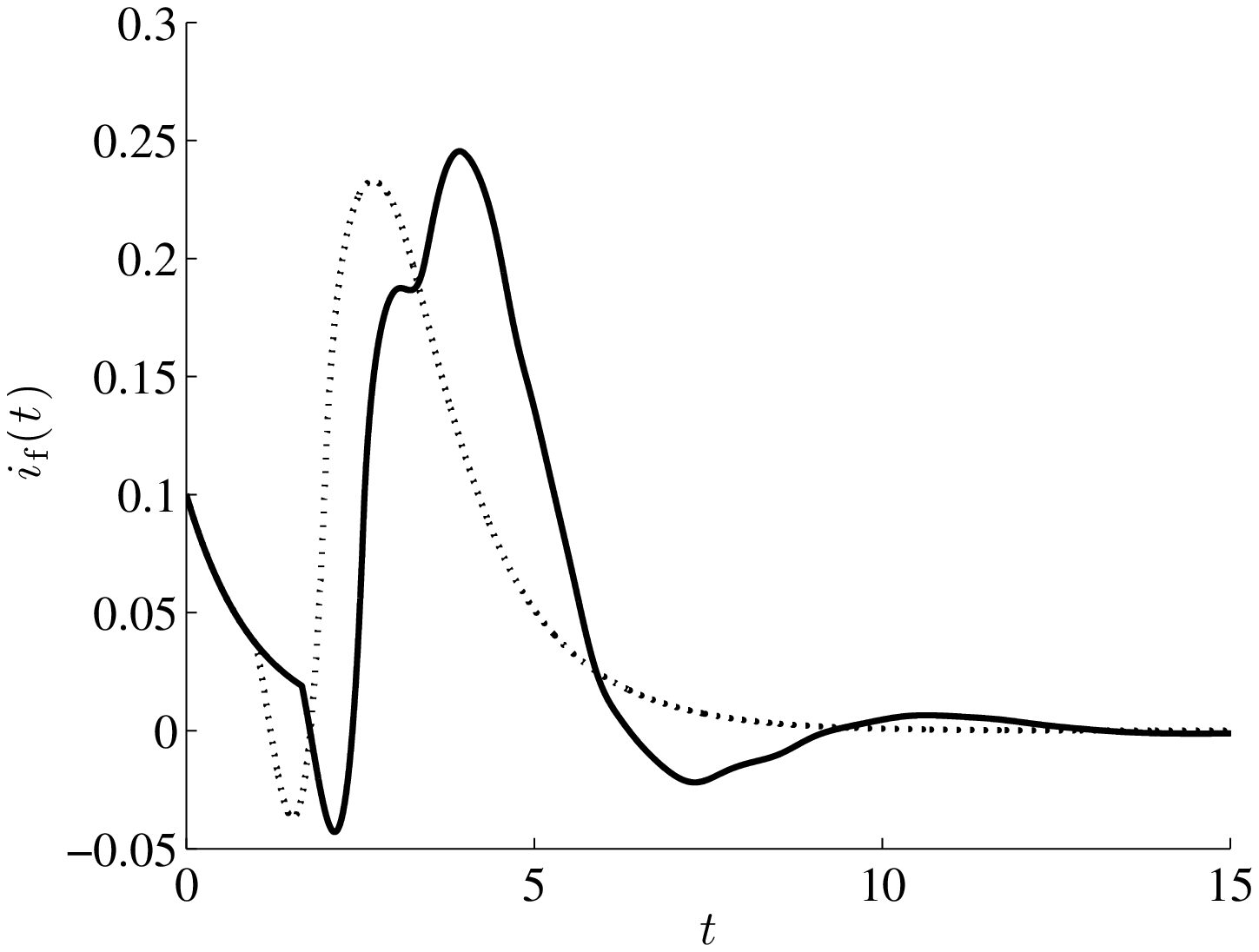}
				\includegraphics[width=.42\textwidth]{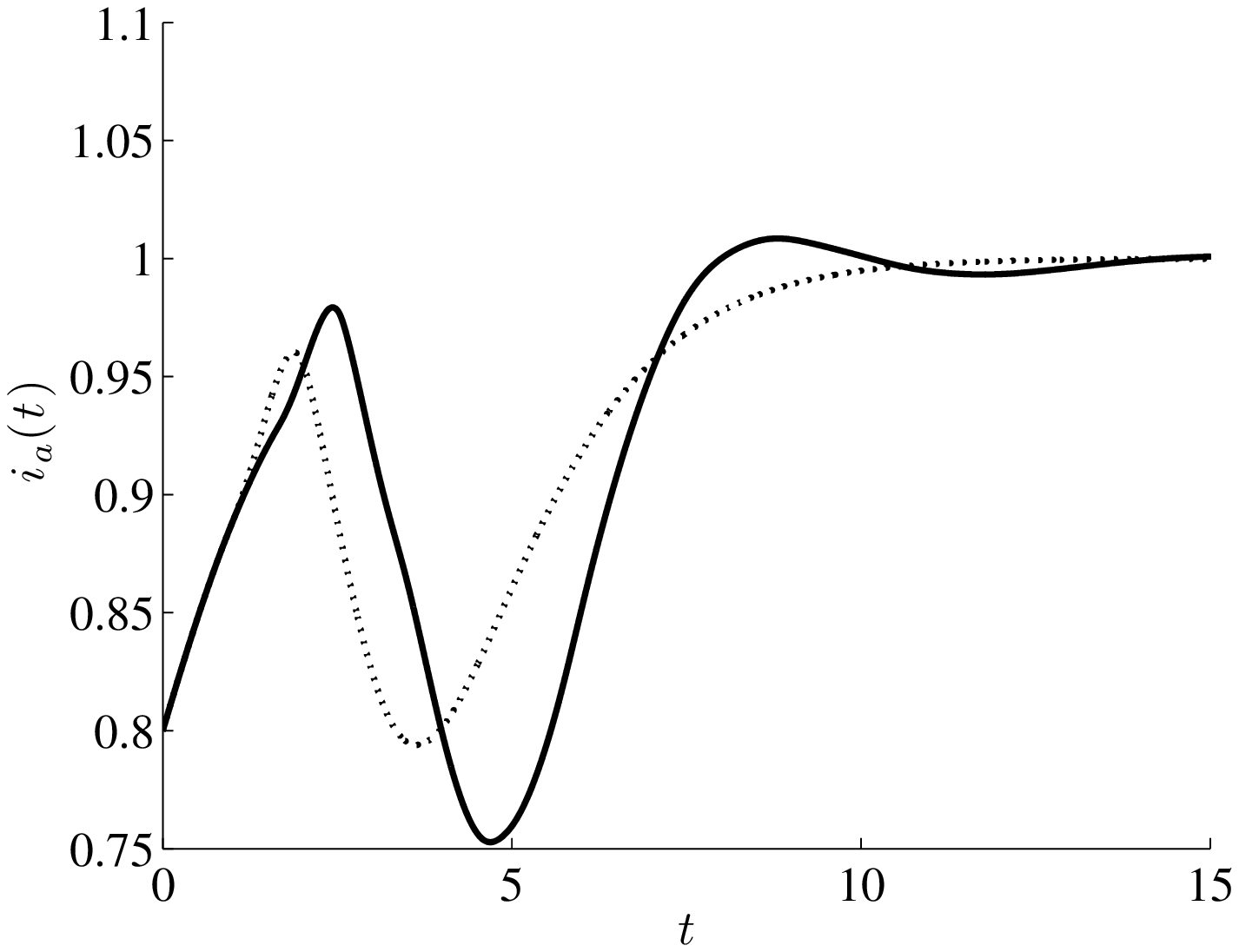}
						\caption{The field (top) and armature (bottom) currents for a network controlled DC motor modeled by (\ref{nomm1})--(\ref{nommn}) with the nominal predictor feedback, under an input delay perturbation $\delta\left(t,i_{\rm a}(t)\right)=0.5i_{\rm a}(t)^2+0.2\sin(t)^2$ (solid line) and $\delta(t,i_{\rm a}(t))=0$ (dotted line). The initial conditions are $i_{\rm f}(0)=0.1$, $i_{\rm a}(0)=0.8$, $\omega(0)=1$ and $U(\theta)=0$, $-1-\delta(0,i_{\rm a}(0))\leq\theta\leq0$.}			
						\label{figi}		
							\end{figure}
				\begin{figure}							
							\centering				
												
												\includegraphics[width=.42\textwidth]{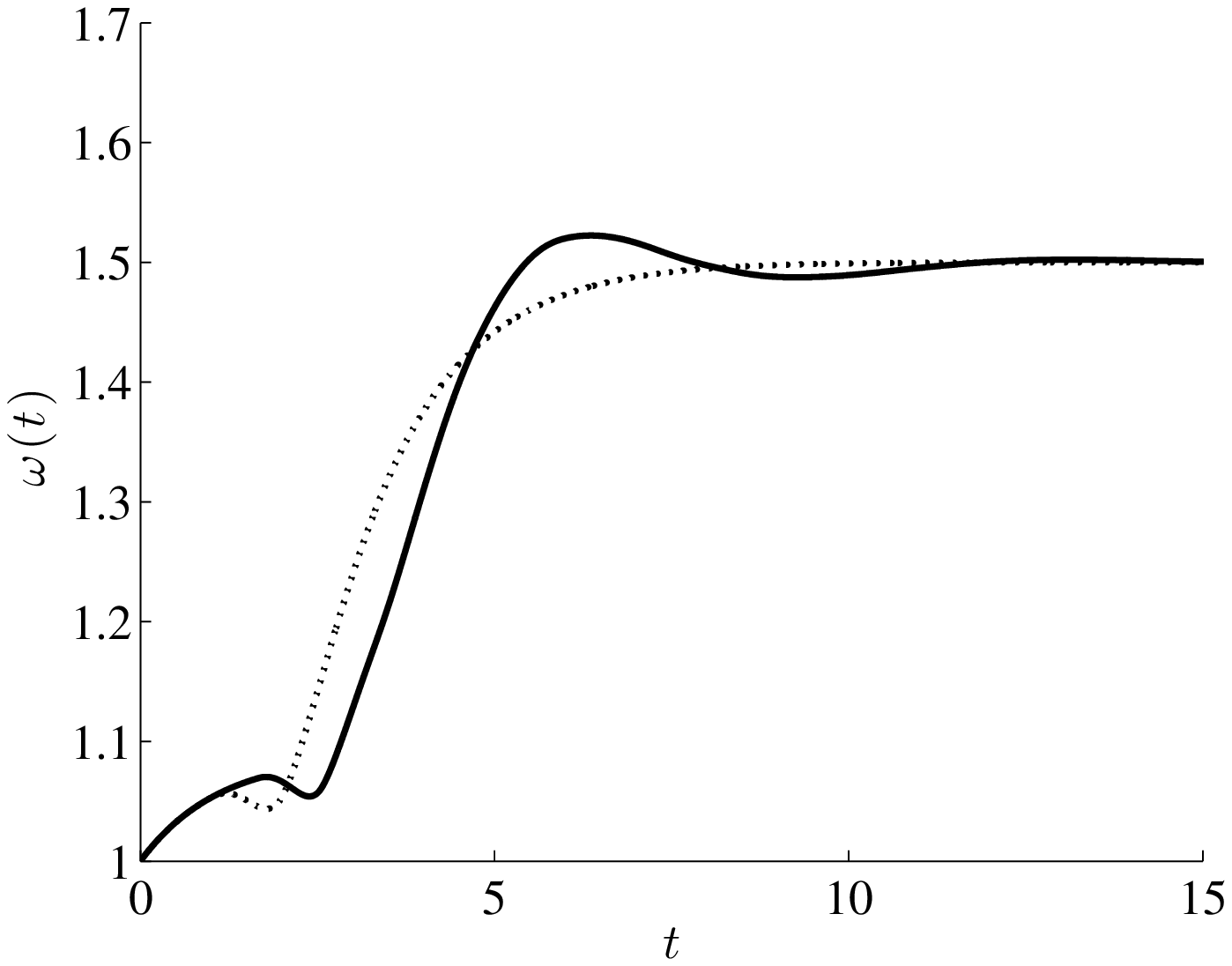}
														\includegraphics[width=.42\textwidth]{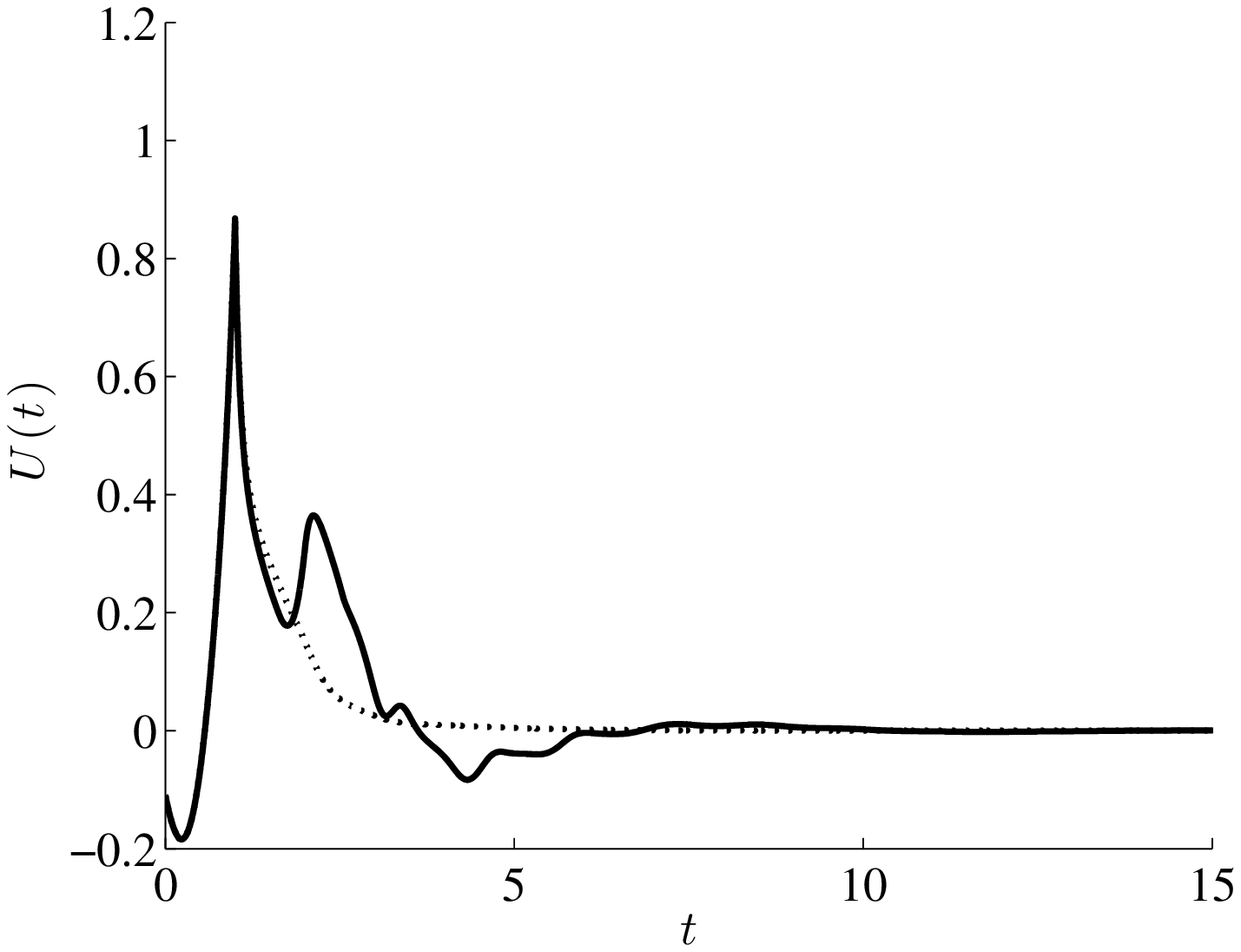}
\caption{The angular velocity (top) and the field voltage (bottom) for a network controlled DC motor modeled by (\ref{nomm1})--(\ref{nommn}) with the nominal predictor feedback, under an input delay perturbation $\delta\left(t,i_{\rm a}(t)\right)=0.5i_{\rm a}(t)^2+0.2\sin(t)^2$ (solid line) and $\delta(t,i_{\rm a}(t))=0$ (dotted line). The initial conditions are $i_{\rm f}(0)=0.1$, $i_{\rm a}(0)=0.8$, $\omega(0)=1$ and $U(\theta)=0$, $-1-\delta(0,i_{\rm a}(0))\leq\theta\leq0$.}
\label{figomega}
\end{figure}

\subsection{Bilateral teleoperation}
In bilateral teleoperation \cite{hoyakem1}, the operator (e.g. a human) controls a robotic system, called the master, at the one end of the communication network. The actions of the master  are transmitted (through the network) to another robotic system, called the slave, at the other end of the network. The goal of the control algorithm is the slave manipulator to behave (in a certain sense) as the master manipulator. A model of two robotic systems, each one having $n$ degrees of freedom, representing the master and the slave manipulators is written as (\cite{hoyakem1})
\begin{eqnarray}
\ddot{x}_{\rm m}(t)+\dot{x}_{\rm m}(t)&=&\tau_{\rm m}\left(t-\hat{D}-\delta(t)\right)\label{sysex1}\\
\ddot{x}_{\rm s}(t)+\dot{x}_{\rm s}(t)&=&\tau_{\rm s}\left(t-\hat{D}-2\delta(t)\right)\label{sysex2},
\end{eqnarray}
where $x_{\rm m}$, $x_{\rm s}$ $\in\mathbb{R}^n$ are the degrees of freedom of the robotic systems and the torques $\tau_{\rm m}$, $\tau_{\rm s}$ $\in \mathbb{R}^n$ are to be designed such as coordination between the master and the slave is achieved asymptotically, i.e., $x_{\rm m}-x_{\rm s}\to0$ as $t\to\infty$. The constant delay $\hat{D}$ represents the known, network-induced delay which is subject to time-varying perturbations that are often present due to congestion, distance etc. \cite{chopra}. For simplicity we assume scalar $x_{\rm m}$, $x_{\rm s}$, $\tau_{\rm m}$, $\tau_{\rm s}$ and we re-write (\ref{sysex1}), (\ref{sysex2}) as
\begin{eqnarray}
\dot{X}_1(t)&=&X_2(t)\\
\dot{X}_2(t)&=&U_1\left(t-\hat{D}-\delta(t)\right)\\
\dot{X}_3(t)&=&X_4(t)\\
\dot{X}_4(t)&=&U_2\left(t-\hat{D}-2\delta(t)\right),
\end{eqnarray}
where $X_1=x_{\rm m}$, $X_2=\dot{x}_{\rm m}$, $X_3=x_{\rm s}$, $X_4=\dot{x}_{\rm s}$, $U_1=\tau_{\rm m}$ and $U_2=\tau_{\rm s}$. A simple controller is (\cite{hoyakem1}) $\tau_{\rm m}(t)=-K_{\rm p}\left(x_{\rm m}(t)-x_{\rm s}(t)\right)-B_{\rm m}\dot{x}_{\rm m}(t)-K_{\rm p}(x_{\rm m}(t)-r)$ and $\tau_{\rm s}(t)=K_{\rm p}\left(x_{\rm m}(t)-x_{\rm s}(t)\right)-B_{\rm s}\dot{x}_{\rm s}(t)-K_{\rm p}\left(x_{\rm s}(t)-r\right)$, where $r$ is the set-point for the positions of the manipulators. The predictor-based version of this controller is
\begin{eqnarray}
U_1(t)&=&-K_{\rm p}\left(\hat{P}_1(t)-\hat{P}_3(t)\right)-B_{\rm m}\hat{P}_2(t)\nonumber\\
&&-K_{\rm p}\left(\hat{P}_1(t)-r\right)\label{pff1}\\
U_2(t)&=&K_{\rm p}\left(\hat{P}_1(t)-\hat{P}_3(t)\right)-B_{\rm s}\hat{P}_4(t)\nonumber\\
&&-K_{\rm p}\left(\hat{P}_3(t)-r\right)\\
\hat{P}_i(t)&=&X_i(t)+\int_{t-\hat{D}}^t\hat{P}_{i+1}(\theta)d\theta,\quad i=1,3\\
\hat{P}_j(t)&=&X_j(t)+\int_{t-\hat{D}}^tU_{\frac{j}{2}}(\theta)d\theta,\quad j=2,4.\label{pff2}
\end{eqnarray}
We choose the desired set-point for $x_{\rm m}$ and $x_{\rm s}$ as $r=2$, the parameters of the controller as $K_{\rm p}=B_{\rm m}=B_{\rm s}=2$, the known delay as $\hat{D}=1$, the initial condition of the plant as $x_{\rm m}(0)=\dot{x}_{\rm m}=\dot{x}_{\rm s}=0$, $x_{\rm s}(0)=1$, the initial actuator state as $U(\theta)=0$, $-\hat{D}-\delta(0)\leq\theta\leq0$ and the initial estimation of the actuator state as $U(\theta)=0$, $-\hat{D}\leq\theta\leq0$. We illustrate the robustness properties of the predictor feedback under a time-varying delay perturbation that is neither in $\mathcal{L}_1$ nor converges to zero as the time goes to infinity nor is small in magnitude. However, after some long period of time its mean is small. This disturbance is show in Fig. \ref{figdis}.

In Fig. \ref{figx} we show the difference between the positions of the master and the slave for the cases where either there is or there is not a perturbation $\delta$. In both cases, under the nominal predictor feedback the position of the slave tracks the position of the master. In Fig. \ref{figup} we show the torques applied to the two robotic systems under the perturbation $\delta$. The control efforts are oscillatory as a result of the effect of the oscillatory perturbation.

\begin{figure}[t]
\centering
		\includegraphics[width=.42\textwidth]{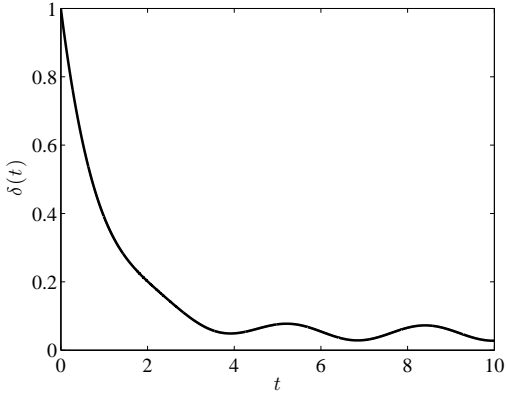}
		\caption{The delay perturbation $\delta$ satisfying $\dot{\delta}(t)=-\delta(t)+0.1\sin(t)^2, \delta(0)=1$ induced by the network in bilateral teleoperation. The model of the two robotic systems is (\ref{sysex1})--(\ref{sysex2}).}
		\label{figdis}
		\end{figure}
		\begin{figure}

												\includegraphics[width=.47\textwidth]{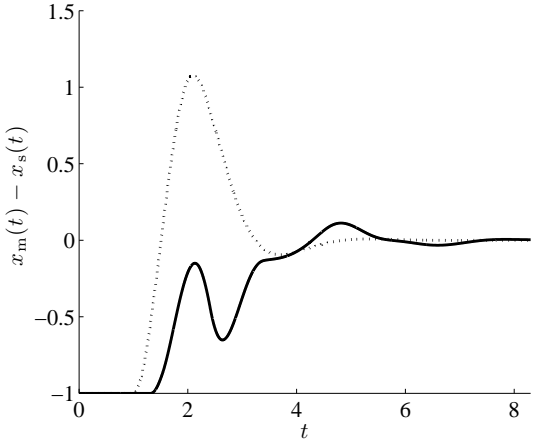}
\caption{The error between the position of the master and the slave manipulators for two robotic systems modeled by (\ref{sysex1})--(\ref{sysex2}). The two manipulators are coordinated through a network with the predictor feedback (\ref{pff1})--(\ref{pff2}), under an input delay perturbation $\delta(t)=0$ (dotted) and $\delta$ satisfying $\dot{\delta}(t)=-\delta(t)+0.1\sin(t)^2$, $\delta(0)=1$ (solid), induced by the network. The initial conditions are $x_{\rm m}(0)=0$, $x_{\rm s}(0)=1$, $\dot{x}_{\rm m}(0)=\dot{x}_{\rm s}(0)=0$ and $\tau_{\rm m}(\theta)=\tau_{\rm s}(\theta)=0$, $-1-\delta(0)\leq\theta\leq0$.}	
\label{figx}
\end{figure}



\begin{figure}[t]
																\includegraphics[width=.45\textwidth]{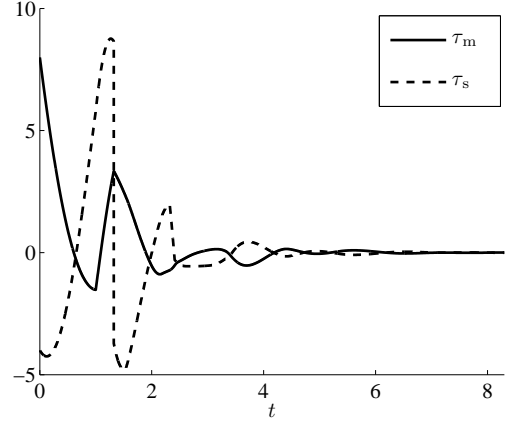}
\caption{The input torques of the master and the slave manipulators for two robotic systems modeled by (\ref{sysex1})--(\ref{sysex2}) coordinated through a network with the predictor feedback (\ref{pff1})--(\ref{pff2}), under an input delay perturbation $\delta$ satisfying $\dot{\delta}(t)=-\delta(t)+0.1\sin(t)^2$, $\delta(0)=1$, induced by the network. The initial conditions are $x_{\rm m}(0)=0$, $x_{\rm s}(0)=1$, $\dot{x}_{\rm m}(0)=\dot{x}_{\rm s}(0)=0$ and $\tau_{\rm m}(\theta)=\tau_{\rm s}(\theta)=0$, $-1-\delta(0)\leq\theta\leq0$.}
\label{figup}
\end{figure}

\section{Conclusions}
Looking into the details of the proofs, we note that in the case of nonlinear systems with state-dependent perturbations, there is a trade off between the achievable region of attraction and the size of the perturbation and its rate, at the origin. For linear systems under time-varying perturbations, global exponential stability holds, but the size of the perturbation and its rate should be appropriately restricted. With the available Lyapunov functional, our next step is to study the inverse optimal redesign problem of predictor feedback for nonlinear systems. 

One might raise the question of robustness to stochastic delay perturbations, since stochastic perturbations have some resemblances with the time-varying case. Yet, in our analysis we restrict not only the magnitude of the perturbation $\delta$ but also the magnitude of its derivative (which also guarantees the invertibility of $\phi=t-\hat{D}-\delta$), which can be unbounded in the case where $\delta$ is white noise, or even when $\delta$ is a low pass version of white noise.

\appendix{\bf Appendices}
\section{The perturbation signals of Lemma \ref{lemma1non}}
\label{secap1}
With $\hat{u}$ defined in (\ref{invback1non}) in terms of $\hat{\rho}$ and $\hat{w}$, the perturbation signals $r_1$, $r$, $r_2$, $r_3$, $r_4$ and $r_5$ are 
\begin{eqnarray}
r_1(x,t)&=&-\hat{D}\frac{\partial \kappa \left(t+\hat{D}x,\hat{\rho}(x,t)\right)}{\partial \hat{\rho}}\nonumber\\
&&\times e^{\hat{D}\int_0^x\frac{\partial f\left(\hat{\rho}(y,t),\hat{u}(y,t)\right)}{\partial \hat{\rho}}dy}\label{brr1}\\
r(x,t)&=& \frac{1}{\hat{D}}  \hat{w}_x(x,t)+\frac{\partial \kappa\left(t+\hat{D}x,\hat{\rho}(x,t)\right)}{\partial t}\nonumber\\
&&+\frac{\partial \kappa\left(t+\hat{D}x,\hat{\rho}(x,t)\right)}{\partial \hat{\rho}} f\left(\hat{\rho}(x,t),\hat{u}(x,t)\right)\label{roonon}\\
r_2(x,t)&\!\!=&\!\!-\hat{D}\left(\hat{D}\frac{\partial^2 \kappa \left(t+\hat{D}x,\hat{\rho}(x,t)\right)}{\partial \hat{\rho}\partial t}\right.\nonumber\\
&&\!\!\left.+f^T\left(\hat{\rho}(x,t),\hat{u}(x,t)\right)\frac{\partial^2 \kappa \left(t+\hat{D}x,\hat{\rho}(x,t)\right)}{\partial \hat{\rho}^2}\right)\nonumber\\
&&\!\!\times e^{\hat{D}\int_0^x\frac{\partial f\left(\hat{\rho}(y,t),\hat{u}(y,t)\right)}{\partial \hat{\rho}}dy}-\hat{D}^2\nonumber\\
&&\!\!\times\frac{\partial \kappa \left(t+\hat{D}x,\hat{\rho}(x,t)\right)}{\partial \hat{\rho}}\frac{\partial f\left(\hat{\rho}(x,t),\hat{u}(x,t)\right)}{\partial \hat{\rho}}\\
r_3(x,t)&=&r_{2_x}(x,t)\\
r_4(t)&=&\hat{D}^2\frac{\partial^2 \kappa \left(t+\hat{D},\hat{\rho}(1,t)\right)}{\partial t\partial \hat{\rho}}e^{\hat{D}\int_0^1\frac{\partial f\left(\hat{\rho}(x,t),\hat{u}(x,t)\right)}{\partial \hat{\rho}}dx}\nonumber\\
&&+\hat{D}^2f^T(\hat{\rho},\hat{u})\frac{\partial^2 \kappa \left(t+\hat{D},\hat{\rho}(1,t)\right)}{\partial^2 \hat{\rho}}\nonumber\\
&&\times e^{\hat{D}\int_0^1\frac{\partial f\left(\hat{\rho}(x,t),\hat{u}(x,t)\right)}{\partial \hat{\rho}}dx}+\frac{\partial \kappa \left(t+\hat{D},\hat{\rho}(1,t)\right)}{\partial \hat{\rho}}\nonumber\\
&&\times e^{\hat{D}\int_0^1\frac{\partial f\left(p(x,t),\hat{u}(x,t)\right)}{\partial p}dx}\hat{D}^2\nonumber\\
&&\times\Biggl(\int_0^1\frac{\partial^2 f\left(\hat{\rho}(x,t),\hat{u}(x,t)\right)}{\partial^2 \hat{\rho}}f(\hat{\rho}(x,t),\hat{u})dx\nonumber\\
&&+\hat{D}\int_0^1\frac{\partial^2 f\left(\hat{\rho}(x,t),\hat{u}(x,t)\right)}{\partial \hat{\rho}\partial \hat{u}}r(x,t)dx\Biggr)\\
r_5(t)&=& {e^{\hat{D}\int_0^1\frac{\partial f\left(\hat{\rho}(y,t),\hat{u}(y,t)\right)}{\partial \hat{\rho}}dy}}^T\frac{\partial^2 \kappa \left(t+\hat{D},\hat{\rho}(1,t)\right)}{\partial^2 \hat{\rho}}\nonumber\\
&&\times\hat{D}^3e^{\hat{D}\int_0^1\frac{\partial f\left(\hat{\rho}(x,t),\hat{u}(x,t)\right)}{\partial \hat{\rho}}dx},\label{brrn}
\end{eqnarray}
where the notation $\frac{\partial^2 f\left(\hat{\rho}(x,t),\hat{u}(x,t)\right)}{\partial^2 \hat{\rho}}f(\hat{\rho}(x,t),\hat{u}(x,t))$ corresponds to a matrix $Q=\left\{g_{i,j}\right\}_{1\leq i,j \leq n}$ with elements $q_{i,j}= \frac{\partial^2 f_i\left(\hat{\rho}(x,t),\hat{u}(x,t)\right)}{\partial \hat{\rho}_j\partial \hat{\rho}}    f(\hat{\rho}(x,t),\hat{u}(x,t))$, where $f=(f_1,\dots, f_n)^T$ and $\hat{\rho}=(\hat{\rho}_1,\ldots,\hat{\rho}_n)^T$.

\section{Technical Lemmas}
\label{appc}
\begin{lemma}
\label{lemdp}
The predictor $p$ in (\ref{defprnon}) satisfies
\begin{eqnarray}
\hat{D}\hat{p}_t(x,t)=\hat{p}_x(x,t)+\hat{D}e^{\hat{D}\int_0^x\frac{\partial f\left(\hat{p}(y,t),\hat{u}(y,t)\right)}{\partial \hat{p}}dy}\tilde{f}(t),
\end{eqnarray}
where $\tilde{f}$ is defined in (\ref{fti}).
\end{lemma}

\begin{proof}
Differentiating (\ref{defprnon}) with respect to $t$, $x$ and using (\ref{obs1non})--(\ref{obs2non}) with the fact that $p(0,t)=X(t)$ we get
\begin{eqnarray}
\Psi(x,t)&=&\hat{D}f(\hat{p}(0,t),\tilde{u}(0,t)+\hat{u}(0,t))\nonumber\\
&&+\hat{D}\int_0^x\frac{\partial f\left(\hat{p}(y,t),\hat{u}(y,t)\right)}{\partial \hat{p}}\hat{D}\hat{p}_t(y,t)dy\nonumber\\
&&+\hat{D}\int_0^x\frac{\partial f\left(\hat{p}(y,t),\hat{u}(y,t)\right)}{\partial \hat{u}}\hat{u}_y(y,t)dy\nonumber\\
&&-\hat{D}f(\hat{p}(x,t),\hat{u}(x,t))\\
\Psi(x,t)&=&\hat{D}\hat{p}_t(x,t)-\hat{p}_x(x,t).
\end{eqnarray}
Since 
\begin{eqnarray}
f(\hat{p}(x,t),\hat{u}(x,t))&=&\int_0^x\left( \frac{\partial f\left(\hat{p}(y,t),\hat{u}(y,t)\right)}{\partial \hat{p}}\hat{p}_y(y,t)\right.\nonumber\\
&&\left.+ \frac{\partial f\left(\hat{p}(y,t),\hat{u}(y,t)\right)}{\partial \hat{u}}\hat{u}_y(y,t)\right)dy\nonumber\\
&&+f(\hat{p}(0,t),\hat{u}(0,t)),
\end{eqnarray}
we get that
\begin{eqnarray}
\Psi(x,t)\!=\!\hat{D}\int_0^x \frac{\partial f\left(\hat{p}(y,t),\hat{u}(y,t)\right)}{\partial \hat{p}}\Psi(y,t)+\hat{D}\tilde{f}(t).\label{eqaman}
\end{eqnarray}
Solving (\ref{eqaman}) for $\Psi$ the lemma is proved.
\end{proof}

\begin{lemma}
\label{lemap1}
There exists a class $\mathcal{K}_{\infty}$ function $\alpha_6$ such that for all $x\in[0,1]$
\begin{eqnarray}
|\hat{p}(x,t)|\leq\alpha_6\left(|X(t)|+\int_0^1\alpha^*(|\hat{u}(x,t)|)dx\right).\label{bounprp}
\end{eqnarray}
\end{lemma}

\begin{proof}
Differentiating (\ref{defprnon}) with respect to $x$ and comparing the resulting ODE with the ODE in $t$ for $X$, the proof is complete with Assumption \ref{ass1} and the comparison principle after appropriately majorizing $e^{\hat{D}}\hat{D}\alpha_3<\alpha^*$. The detailed proof can be found in \cite{krstic feed} (Lemma 7).
\end{proof}

\begin{lemma}
\label{lemmauw}
There exists class $\mathcal{K}_{\infty}$ function $\alpha_{11}\ldots\alpha_{13}$ such that for all $x\in[0,1]$
\begin{eqnarray}
|\hat{w}(x,t)|&\leq&\alpha_{11}\left(\Omega(t)\right)\label{21}\\
|\hat{w}_x(x,t)|&\leq&|\hat{u}_x(x,t)|+\alpha_{12}\left(\Omega(t)\right)\label{22}\\
\int_0^1\hat{w}_{xx}(x,t)^2dx&\leq&6\int_0^1\hat{u}_{xx}(x,t)^2dx+\alpha_{13}\left(\Omega(t)\right),
\end{eqnarray}
where
\begin{eqnarray}
\Omega(t)\!=\!|X(t)|\!+\!\int_0^1\alpha^*(|\hat{u}(x,t)|)dx\!+\!\!\int_0^1\!\!\hat{u}_x(x,t)^2dx.
\end{eqnarray}
\end{lemma}

\begin{proof}
The proof of the lemma is based on algebraic manipulations and routine class $\mathcal{K}$ majorizations using the direct (\ref{back1non}) backstepping transformation together with relations (\ref{defprnon}) for the predictor state and Lemma \ref{lemap1}. For the reader's benefit we prove (\ref{21}) and (\ref{22}). The rest can be proved similarly. From (\ref{back1non}) and (\ref{kap}) we get that $|\hat{w}(x,t)|\leq|\hat{u}(x,t)| +\hat{\alpha}\left(|\hat{p}(x,t)|\right)$. Using the fact that 
\begin{eqnarray}
\sup_{x\in[0,1]}|\hat{u}(x,t)|\leq |\hat{u}(1,t)|+\int_0^1|\hat{u}_x(x,t)|dx,\label{fgg}
\end{eqnarray}
 with relation (\ref{connl}) and Lemma \ref{lemap1} we get (\ref{21}). For proving (\ref{22}) we proceed as follows. Differentiating (\ref{back1non}) with respect to $x$ we get that 
\begin{eqnarray}
\hat{w}_x(x,t)&\!=\!&\hat{u}_x(x,t)+\hat{D}\frac{\partial \kappa\left(t+\hat{D}x,\hat{p}(x,t)\right)}{\partial t}+\hat{D}\nonumber\\
&&\times\frac{\partial \kappa\left(t+\hat{D}x,\hat{p}(x,t)\right)}{\partial \hat{p}}\!f\left(\hat{p}(x,t),\hat{u}(x,t)\right)\!.
\end{eqnarray}
Combining (\ref{kap}), (\ref{fd}) with Lemma \ref{lemap1} and (\ref{fgg}) we arrive at (\ref{22}) with appropriate class $\mathcal{K}$ majorizations.
\end{proof}

\begin{lemma}
\label{lemz}
There exists positive constants $M^*$, $c^*$ such that for all solutions of the system satisfying (\ref{222}) the following holds for all $x\in[0,1]$
\begin{eqnarray}
|\hat{\rho}(x,t)|\leq M^*\left(|X(t)|+\int_0^1\alpha^*(|\hat{w}(x,t)|)dx\right).
\end{eqnarray}
\end{lemma}

\begin{proof}
Under Assumtpion \ref{ass2} and choosing $c^*<R$, from Theorem 4.14 from \cite{khalil} there exist a continuously differentiable function $S:[t_0,\infty)\times D_R\to\mathbb{R}^n$, where $D_R=\left\{X\in\mathbb{R}^n||X|<R\right\}$, and positive constants $M_1$, $M_2$, $M_3$ and $M_4$ such that for all $X\in D_R$
\begin{eqnarray}
M_1|X|^2&\leq& S(t,X)\leq M_2|X|^2\label{issn1}\\
\frac{\partial S\left(t,X(t)\right)}{\partial t}&+&\frac{\partial S\left(t,X(t)\right)}{\partial X}f\left(X(t),\kappa\left(t,X(t)\right)\right)\nonumber\\
&\leq& -M_3|X(t)|^2\label{iss1}\\
\left|\frac{\partial S\left(t,X\right)}{\partial X}\right|&\leq& M_4|X|.\label{issnew}
\end{eqnarray}
Since $f\in C^2\left(\mathbb{R}^n\times\mathbb{R};\mathbb{R}\right)$, for all $X\in D_R$ and every $\omega\in\mathbb{R}$ such that $|\omega|\leq M$ for some positive constant $M$, there exists an increasing function in both arguments $L\in C\left(\mathbb{R}_+^2;\mathbb{R}_+\right)$ such that along the solutions of $\dot{X}(t)=f\left(X(t),\kappa\left(t,X(t)+\omega(t)\right)\right)$ it holds that
\begin{eqnarray}
\dot{S}\!&\leq\!& -M_3|X(t)|^2\!+\! \frac{\partial S\left(t,X\right)}{\partial X}\left(f\left(X(t),\kappa\left(t,X(t)\right)\!+\!\omega(t)\right)\right.\nonumber\\
&&\left.-f\left(X(t),\kappa\left(t,X(t)\right)\right)\right)\nonumber\\
&\leq&-M_3|X(t)|^2+ M_4|X(t)|L(R,M)|\omega(t)|,
\end{eqnarray}
where we used Lemma 3.1 in \cite{khalil}. With $S^*=\sqrt{S}$ we get
\begin{eqnarray}
\dot{S^*}(t,X(t))\!\leq\! -\frac{M_3}{2\sqrt{M_1}}|X(t)|\! +\! \frac{M_4L(R,M)}{2\sqrt{M_1}}|\omega(t)|.\label{iss1new}
\end{eqnarray}
Differentiating (\ref{invdefprnon}) with respect to $x$ we get that
\begin{eqnarray}
\hat{\rho}_x(x,t)&=&\hat{D}f\left(\hat{\rho}(x,t),\kappa\left(t+\hat{D}x,\hat{\rho}(x,t)\right)\right.\nonumber\\
&&+\left.\hat{w}(x,t)\right).\label{fht}
\end{eqnarray}
Using the fact that for all $x\in[0,1]$, $|\hat{w}(x,t)|\leq \int_0^1|\hat{w}_x(x,t)|dx$ (which follows from (\ref{tilw33})), relation (\ref{22}) together with (\ref{222}) and Lemma \ref{lemga} give that for all $x\in[0,1]$, $|\hat{w}(x,t)|\leq M$, with $M=\zeta_2(c^*)+\alpha_{12}\left(\zeta_2(c^*)\right)$. Using a change of variables in (\ref{fht}) as $x'={t+\hat{D}x}$ and comparing the resulting ODE in $x'$ for $\hat{\rho}$ with the ODE in $t$ for $\dot{X}(t)=f\left(X(t),\kappa(t,X(t)+\omega(t)\right)$, with (\ref{issn1}) and after appropriately majorizing $s<\alpha^*(s)$, the proof is complete with $M^*(R,M)=\frac{\sqrt{M_2}}{\sqrt{M_1}}+\hat{D} \frac{M_4L(R,M)}{2{M_1}}$, and hence, with $c^*=\min\left\{R,c_1^*\right\}$, where $c_1^*$ satisfies 
\begin{eqnarray}
M^*(R,M(c_1^*))\left(c_1^*+\alpha_{11}\left(\zeta_2(c_1^*)\right)\right)<R.\label{cshit}
\end{eqnarray}
\end{proof}

\begin{lemma}
\label{lemmauwnew}
There exists class $\mathcal{KC}_{\infty}$ functions $\alpha_{14}\ldots\alpha_{16}$ and a positive constant $c^*$ such that for all solutions of the systems satisfying (\ref{222}), the following holds
\begin{eqnarray}
|\hat{u}(x,t)|&\leq&\alpha_{14}\left(Y(t),R\right)\\
|\hat{u}_x(x,t)|&\leq&|\hat{w}_x(x,t)|+\alpha_{15}\left(Y(t),R\right)\\
\int_0^1\hat{u}_{xx}(x,t)^2&\leq&6\int_0^1\hat{w}_{xx}(x,t)^2dx\!+\!\alpha_{16}\left(Y(t),R\right)\!,
\end{eqnarray}
 for all $x\in[0,1]$, where
\begin{eqnarray}
Y(t)\!\!=\!\!|X(t)|\!+\!\int_0^1\alpha^*(|\hat{w}(x,t)|)dx\!+\!\!\int_0^1\!\!\hat{w}_x(x,t)^2dx.\label{defY}
\end{eqnarray}
\end{lemma}

\begin{proof}
Choose $c^*$ as in Lemma \ref{lemz}. Then, the proof of the lemma is based on algebraic manipulations and routine class $\mathcal{K}$ majorizations using the inverse transformation (\ref{invback1non}), relation (\ref{invdefprnon}) for the predictor state and Lemma \ref{lemz}.
\end{proof}

\begin{lemma}
\label{brfo}
There exist class $\mathcal{KC}_{\infty}$ functions $\alpha_{17}\ldots\alpha_{23}$ and positive constants $c^*, \mu_1, \mu_2, \mu_3, \mu_4, \mu_5$ such that for all solutions of the system satisfying (\ref{222}) the following holds for all $x\in[0,1]$
\begin{eqnarray}
|{r}(x,t)|&\leq&\frac{1}{\hat{D}}|\hat{w}_x(x,t)|+\alpha_{17}\left(Y(t),R\right)\label{rr00}\\
|r_1(x,t)|&\leq&\mu_1+\alpha_{18}\left(Y(t),R\right)\\
|r_2(x,t)|&\leq&\mu_2+\alpha_{19}\left(Y(t),R\right)\\
\int_0^1{r}_3(x,t)^2dx&\leq&\mu_3+\alpha_{20}\left(Y(t),R\right)\\
|{r}_4(t)|&\leq&\mu_4+\alpha_{21}\left(Y(t),R\right)\\
|r_5(t)|&\leq&\mu_5+\alpha_{22}\left(Y(t),R\right)\\
\int_0^1{r}_x(x,t)^2dx&\leq&\alpha_{23}\left(Y(t),R\right)\nonumber\\
&&+\frac{6}{\hat{D}^2}\int_0^1\hat{w}_{xx}(x,t)^2dx,\label{rrnn}
\end{eqnarray}
where $Y(t)$ is defined in (\ref{defY}).
\end{lemma}

\begin{proof}
Let $c^*$ be as in Lemma \ref{lemz}. The proof is based on (\ref{brr1})--(\ref{brrn}) combined with (\ref{kap}), the fact that $f$ is twice differentiable and with similar calculations as in the proof of Lemma \ref{lemmauw}. Yet, we provide the proofs of (\ref{rr00}), (\ref{rrnn}) as some of the steps are useful later on. Under Assumption \ref{ass2} (which allows us to choose $\hat{\alpha}$ continuously differentiable without loss of generality), Lemma \ref{lemz}, and the facts that, $|\hat{w}(x,t)|\leq\int_0^1|\hat{w}_x(x,t)|dx$, for all $x\in[0,1]$ (which follows from (\ref{tilw33})), and that $f:C^2\left(\mathbb{R}^n\times\mathbb{R};\mathbb{R}^n\right)$, $f(0,0)\!=\!0$, which allows us to conclude
\begin{eqnarray}
\left|f\left(X,\omega\right)\right|\leq\alpha_5\left(|X|+|\omega|\right),\label{fd}
\end{eqnarray}
for some function $\alpha_5\in\mathcal{K}_{\infty} \cap C^1$, we get from (\ref{roonon}) that
\begin{eqnarray}
|{r}(x,t)|&\leq&\frac{1}{\hat{D}}|\hat{w}_x(x,t)|+\alpha_{r}\left(\Lambda(t),R\right)\label{rr00new}
\end{eqnarray}
\begin{eqnarray}
\Lambda(t)&=&|X(t)|+\int_0^1\alpha^*(|\hat{w}(x,t)|) dx\nonumber\\
&&+\int_0^1|\hat{w}_x(x,t)|dx,\label{deflambda}
\end{eqnarray}
for some class $\mathcal{KC}_{\infty}$ function $\alpha_r$, continuously differentiable in its first argument. Analogously, differentiating (\ref{roonon}) with respect to $x$ and using  (\ref{invdefprnon}) together with the fact that $\hat{D}r(x,t)=\hat{u}_x(x,t)$ it is shown that
\begin{eqnarray}
|{r}_x(x,t)|&\leq&\frac{1}{\hat{D}}|\hat{w}_{xx}(x,t)|+\alpha_{1,r_x}\left(\Lambda(t),R\right)\nonumber\\
&&+\left(\mu^*+\alpha_{2,r_x}\left(\Lambda(t),R\right)\right)|r(x,t)|,\label{rrnnnew}
\end{eqnarray}
for some positive constant $\mu^*$ and some functions $\alpha_{1,r_x}$, $\alpha_{2,r_x}$ $\in\mathcal{KC}_{\infty}$ which are continuously differentiable with respect to their first argument. With the Cauchy-Schwarz inequality we get (\ref{rr00}), (\ref{rrnn}).
\end{proof}

\section{Proof of Lemma \ref{lemmanm}}
\label{secprb}
The proof of this lemma is based on the following fact.

\begin{fact}
\label{lemforp}
There exists a class $\mathcal{KC}_{\infty}$ function $\hat{\zeta}_1$ such that if the perturbation $\delta$ and the solutions of the system satisfying (\ref{prhx}) for $0<c<1$ and (\ref{prhxnew}), the following holds for all $\phi(t)\leq\theta\leq t$
\begin{eqnarray}
\left|P^*(\theta)\right|&\leq&\hat{\zeta}_1\left(\left|X(t)\right|+ \sup_{\phi(t)\leq s\leq t}\left|U(s)\right|,R\right) \label{resr2gl}.
\end{eqnarray}
\end{fact}

\begin{proof}
The proof can be found in \cite{bek} (Lemma 4).
\end{proof}

It holds $U(\theta)=U(t)-\int_{\theta}^t\dot{U}(s)ds$, for all $\phi(t)\leq\theta\leq t$, and hence, using (\ref{kap}), (\ref{connl}) and the fact that $\int_{\theta}^t\dot{U}(s)ds=\int_{\frac{\sigma(\theta)-t}{\sigma(t)-t}}^1u_x(x,t)dx$ we get $\sup_{\phi(t)\leq \theta\leq t}\left|U(\theta)\right|\!\leq\! \hat{\alpha}\left(|\hat{p}(1,t)|\right)$ $+\int_0^1\left(|\hat{u}_x(x,t)|+|\tilde{u}_x(x,t)|\right)dx$. Lemma \ref{lemap1} (App. B) and the Cauchy-Schwarz give
\begin{eqnarray}
\sup_{\phi(t)\leq \theta\leq t}\!\left|U(\theta)\right|&\!\leq\!&\alpha_{4}\left(|X(t)|\!+\!\int_0^1\alpha^*(|\hat{u}(x,t)|)dx\right.\nonumber\\
&&\!\left.\!+\!\int_0^1\hat{u}_x(x,t)^2dx\!+\!\!\int_0^1\!\tilde{u}_x(x,t)^2dx\right)\!,\label{rebn}
\end{eqnarray}
 for some class $\mathcal{K}_{\infty}$ function $\alpha_4$. Using (\ref{89}), (\ref{fd}), conditions (\ref{prhx}) for $0<c<1$ and (\ref{prhxnew}) are satisfied, if the following relation is satisfied for all $\phi(t)\leq\theta\leq t$
\begin{eqnarray}
R_1&>&c_1+\hat{\mu}(|P^*(\theta)|)+\left(c_1+\hat{\mu}(|P^*(\theta)|)\right)\nonumber\\
&&\times\alpha_5\left(|X(t)|+\sup_{\phi(t)\leq\theta\leq t}|U(\theta)|\right),
\end{eqnarray}
where $R_1=\min\left\{c,\hat{D}\right\}$. With Fact \ref{lemforp}, Lemma \ref{lemga} and (\ref{rebn}) the lemma is proved with $c^*=c_2^*$ and $c_1$ satisfying
\begin{eqnarray}
R&>&\left(c_1+\hat{\mu}\left(\hat{\zeta}_1\left(\alpha_{4}\left(3{\zeta}_{2}\left({c}_2^*,R\right)\right),R\right)\right)\right)\nonumber\\
&&\times \left(1+\alpha_5\left(\alpha_{4}\left(3{\zeta}_{2}\left({c}_2^*,R\right)\right)\right)\right).\label{ccc}
\end{eqnarray}

\section{Proof of Lemma \ref{lyapl}}
\label{apply}
Let $c^*$ be the minimum of $c_1^*$ and $c_2^*$ defined in (\ref{cshit}) and (\ref{ccc}) respectively. Taking the derivative of $V$ with $S^*=\sqrt{S}$ and using integration by parts together with (\ref{tilw1})--(\ref{tilw2}), (\ref{obs3non1})--(\ref{obs3non}), (\ref{tilu1})--(\ref{tilu2}), (\ref{trr1non})--(\ref{sys2ntr11nonxx}), (\ref{iss1new}) we get
\setlength{\arraycolsep}{0pt}\begin{eqnarray}
\dot{V}(t)&\leq&-\frac{M_3}{2\sqrt{M_1}}|X(t)|+\frac{M_4L^*(R)}{2\sqrt{M_1}}|\hat{w}(0,t)|\nonumber\\
&&+\frac{M_4L^*(R)}{2\sqrt{M_1}}|\tilde{u}(0,t)|-g_{11}\pi(0,t)|\tilde{u}(0,t)|\nonumber\\
&&-g_1g_{11}  \int_0^1e^{g_1x}\pi(x,t)|\tilde{u}(x,t)|dx-g_{11}\pi_x(x,t)\nonumber\\
&&\times \int_0^1e^{g_1x}|\tilde{u}(x,t)|dx+g_{11}\sup_{x\in[0,1]}\left|1-\hat{D}\pi(x,t)\right|\nonumber\\
&&\times\int_0^1e^{g_1x}|r(x,t)|dx+g_6e^{g_2}\left|\frac{1}{\pi(1,t)}-\hat{D}\right|^2\nonumber\\
&&\times r(1,t)^2  -g_6\pi(0,t)\tilde{u}_x(0,t)^2-g_2g_6\nonumber\\
&&\times\int_0^1e^{g_2x}\pi(x,t)\tilde{u}_x(x,t)^2dx+ g_6\pi_x(x,t)\nonumber\\
&&\times\int_0^1e^{g_2x}\tilde{u}_x(x,t)^2dx +2g_6\sup_{x\in[0,1]}\left|1-\hat{D}\pi(x,t)\right|\nonumber\\
&&\times\int_0^1e^{g_2x}|r_x(x,t)||\tilde{u}_x(x,t)|dx-g_7\hat{w}_x(0,t)^2\nonumber\\
&&-g_4g_7\int_0^1e^{g_4x}\hat{w}_x(x,t)^2dx-g_9g_{10}\nonumber\\
&&\times\int_0^1e^{g_{10}x}|\hat{w}_x(x,t)|dx-g_9|w_x(0,t)|+2g_7\nonumber\\
&&\times\int_0^1e^{g_4x}|\hat{w}_x(x,t)| |r_2(x,t)|dx\left|\tilde{f}(t)\right|\nonumber\\
&&+g_9\int_0^1e^{g_{10}x}|r_2(x,t)|dx|\tilde{f}(t)|+g_7e^{g_4}r_1(1,t)^2\nonumber\\
&&\times\left|\tilde{f}(t)\right|^2+g_9e^{g_{10}}|r_1(1,t)||\tilde{f}(t)|+g_8e^{g_5}\nonumber\\
&&\times\hat{w}_{xx}(1,t)^2-g_8\hat{w}_{xx}(0,t)^2-g_8g_5\int_0^1e^{g_5x}\nonumber\\
&&\times\hat{w}_{xx}(x,t)^2dx+2g_8\int_0^1e^{g_5x} |r_3(x,t)|\nonumber\\
&&\times|\hat{w}_{xx}(x,t)|dx\left|\tilde{f}(t)\right|-g_{12}\alpha^*\left(|\hat{w}(0,t)|\right)\nonumber\\
&&-g_{12}g_3\int_0^1\alpha^*\left(|\hat{w}(x,t)|\right)dx+g_{12}\nonumber\\
&&\times\int_0^1e^{g_3x}|\alpha{^*}'(|\hat{w}(x,t)|)||r_1(x,t)|dx\left|\tilde{f}(t)\right|,\label{rel1}
\end{eqnarray}\setlength{\arraycolsep}{5pt}for an increasing function $L^*\in C\left(\mathbb{R}_+;\mathbb{R}_+\right)$. Using (\ref{defpinon}) and Lemma \ref{lemmanm} we get for all $x\in[0,1]$, $\frac{1}{\left(1+R\right)\left(\hat{D}+R\right)}\leq\pi(x,t)\leq\frac{1}{\left(1-R\right)\left(\hat{D}-R\right)}$, $|\pi_x(x,t)|\leq\frac{\hat{c}}{\left(1-R\right)\left(\hat{D}-R\right)}$, where 
\begin{eqnarray}
\hat{c}&=&\left(c_1+\hat{\mu}\left(\hat{\zeta}_1\left(\alpha_{4}\left(3{\zeta}_{2}\left(\min\left\{c_1^*,{c}_2^*\right\},R\right)\right),R\right)\right)\right)\nonumber\\
&&\times\left(1+\alpha_5\left(\alpha_{4}\left(3{\zeta}_{2}\left(\min\left\{c_1^*,{c}_2^*\right\},R\right)\right)\right)\right).\label{tosp}
\end{eqnarray}
 Moreover, since $\pi(x,t)$ is linear in $x$, it takes its maximum value either at $x=0$ or at $x=1$, and hence, 
\begin{eqnarray}
\left|{1}-\hat{D}\pi(x,t)\right|&\leq&\max\left\{\left|{1}-\hat{D}\pi(0,t)\right|,\left|{1}-\hat{D}\pi(1,t)\right|\right\}\nonumber\\
&\leq&M_2Z(t),
\end{eqnarray}
where $M_2=\pi(0,t)+\pi(1,t)\leq\frac{2}{\left(1-R\right)\left(\hat{D}-R\right)}$, and $Z(t)=\max\left\{\left|\delta\left(\sigma(t),P^*(t)\right)\right|,\left|\delta\left(\sigma(t),P^*(t)\right)+\hat{D}\right|\right.$ $\left.\times\!\left( \left| \delta_t\left(\sigma(t),P^*(t)\right)\right|\!+\!\left|\nabla\delta\left(\sigma(t),P^*(t)\right)\!f\left(P^*(t),U(t)\right)\right|\right)\right\}$. Therefore, using (\ref{ccc}) we have that
\begin{eqnarray}
\sup_{x\in[0,1]}\left|{1}-\hat{D}\pi(x,t)\right|&\leq&2\hat{c}B^*(R)\\
B^*(R)&=&\frac{1}{\hat{D}-R}\times\frac{\hat{D}+2}{1-R}.
\end{eqnarray}
Since $f:C^2\left(\mathbb{R}^n\times\mathbb{R};\mathbb{R}^n\right)$, using relations (\ref{fti}), (\ref{fzz}) and (\ref{fuu}), with Lemma 3.1 from \cite{khalil} and (\ref{222}) we have that
\begin{eqnarray}
\left|\tilde{f}(t)\right|+\left|\tilde{f}_{\hat{\rho}}(t)\right|+\left|\tilde{f}_{\hat{u}}(t)\right|&\leq& \kappa_1(R)\left|\tilde{u}(0,t)\right|,\label{brr}
\end{eqnarray} 
for an increasing function $\kappa_1\in C\left(\mathbb{R}_+;\mathbb{R}_+\right)$. Hence, from (\ref{sys2ntr11non}), (\ref{rr00new}) we conclude after using (\ref{brr}), (\ref{deflambda}) that
\begin{eqnarray}
r(1,t)^2&\leq& \kappa_{2}(R)|\tilde{u}(0,t)| +2\alpha_{r}^2\left(\Lambda(t),R\right),\label{kl1}
\end{eqnarray}
for an increasing function $\kappa_2\in C\left(\mathbb{R}_+;\mathbb{R}_+\right)$ (where we also used the fact that $\left|\tilde{f}\right|^2\leq c(R) \left|\tilde{f}\right|$ which follows form (\ref{fti}) and (\ref{222})). We are concerned next with $\hat{w}_{xx}(1,t)^2$. With Young's inequality and (\ref{rr00}), from (\ref{sys2ntr11nonxx}) we get that there exists an increasing function $\kappa_3\in C\left(\mathbb{R}_+;\mathbb{R}_+\right)$ such that (where we absorb the powers of $\left|\tilde{f}\right|$, $\left|\tilde{f}_{\hat{\rho}}\right|$ and $\left|\tilde{f}_{\hat{u}}\right|$ higher than one in $\kappa_3$ based on (\ref{fti}), (\ref{fzz}), (\ref{fuu}), (\ref{222}))
	\begin{eqnarray}
	\hat{w}_{xx}(1,t)^2&\leq& \kappa_3(R)\left(\left|\tilde{f}(t)\right|+\left|\tilde{f}_{\hat{\rho}}(t)\right|+\left|\tilde{f}_{\hat{u}}(t)\right|\right)\nonumber\\
	&&+\kappa_3(R)\tilde{u}_x(0,t)^2+\kappa_3(R)\hat{w}_x(0,t)\nonumber\\
	&&+\left(\sup_{x\in[0,1]}\left|1-\hat{D}\pi(x,t)\right|\right)^2\kappa_3(R) \nonumber\\
	&&\times\alpha_{r}^2\left(\Lambda(t),R\right),\label{kl3}
	\end{eqnarray}
	where we used (\ref{deflambda}) and the fact that $r(0,t)^2\!\leq\!\frac{2}{\hat{D}^2}\hat{w}_x(0,t)^2+2\alpha_{r}^2\left(\Lambda(t),R\right)$ which follows from (\ref{rr00}). 
From  (\ref{deflambda}), (\ref{rrnnnew}) and (\ref{222}) we get 
\setlength{\arraycolsep}{0pt}\begin{eqnarray}
\int_0^1{r}_x(x,t)^2dx&\leq&\frac{3}{\hat{D}^2}\int_0^1\hat{w}_{xx}(x,t)^2dx+\kappa_4(R)\nonumber\\
&&\times \left(\int_0^1\hat{w}_x(x,t)^2dx+\alpha_{r}^2\left(\Lambda(t),R\right)\right)\nonumber\\
&&+3\alpha_{1,r_x}^2\left(\Lambda(t),R\right).\label{rrnnnew1}
\end{eqnarray}\setlength{\arraycolsep}{5pt}With relation (\ref{222}) and Lemmas \ref{lemag}, \ref{lemga}, \ref{brfo} (App. B), from (\ref{rel1}) one can conclude that the terms that multiply $|\tilde{f}|$ are bounded by $\left(g_7e^{g_4}+g_9e^{g_{10}}+g_8e^{g_5}+g_{12}e^{g_3}\right)\kappa_5(R)$. Choosing  $g_1\!=\!g_2\!=\!\left(1+R\right)\left(\hat{D}+R\right)\left(1+\frac{R}{\left(1-R\right)\left(\hat{D}-R\right)}\right)$ and combining (\ref{kl1}), (\ref{kl3}), (\ref{rrnnnew1}), (\ref{brr}) with Young's inequality, we get from (\ref{rel1}), (\ref{rr00new}), (\ref{deflambda}) and (\ref{rrnnnew})
\setlength{\arraycolsep}{0pt}\begin{eqnarray}
\dot{V}(t)&\leq&-\frac{M_3}{2\sqrt{M_1}}|X(t)|+\frac{M_4L^*(R)}{2\sqrt{M_1}}\left(|\hat{w}(0,t)|+|\tilde{u}(0,t)|\right)\nonumber\\
&&-g_{11}  \int_0^1|\tilde{u}(x,t)|dx-g_{12}\alpha^*\left(|\hat{w}(0,t)|\right) \nonumber\\
&&-{g_{11}}{\left(1+R\right)^{-1}\left(\hat{D}+R\right)}^{-1}|\tilde{u}(0,t)|-\tilde{u}_x(0,t)^2\nonumber\\
&& \times\left({g_6}{\left(1+R\right)^{-1}\left(\hat{D}+R\right)^{-1}}-g_8e^{g_5}\kappa_3(R)\right)\nonumber\\
&&-\left(g_7-g_8e^{g_5}\kappa_3(R)\right)\hat{w}_x(0,t)^2-g_6\nonumber\\
&&\times\left(1-2\hat{c}B^*(R)\right)\int_0^1e^{g_2x}\tilde{u}_x(x,t)^2dx \nonumber\\
&&-g_{12}g_3\int_0^1\alpha^*\left(|\hat{w}(x,t)|\right)dx-g_8\hat{w}_{xx}(0,t)^2\nonumber\\
&&-\left(g_7g_4-{2\hat{c}e^{g_2}{g_6}\left(\hat{D}\kappa_4(R)+1\right)}{\hat{D}}^{-1} B^*(R)\right)\nonumber\\
&&\times\int_0^1e^{g_4x}\hat{w}_x(x,t)^2dx-\left(g_8g_5-{g_6e^{g_2}}{\hat{D}}^{-2}\right.\nonumber\\
&&\left.\times 6\hat{c}B^*(R)\right)\int_0^1\hat{w}_{xx}(x,t)^2dx\nonumber\\
&&+\left(4g_6e^{g_2}\left(\hat{D}+1\right)^2\left|1-\hat{D}\pi(1,t)\right|^2\kappa_2(R)\right.\nonumber\\
&&\left.+\left(g_7e^{g_4}+g_9e^{g_{10}}+g_8e^{g_5}+g_{12}e^{g_3}\right)\kappa_5(R)\right)\nonumber\\
&&\times|\tilde{u}(0,t)|-\left(g_9g_{10}-{g_{11}e^{g_1}2\hat{c}B^*(R)}{\hat{D}}^{-1}\right)\nonumber\\
&&\times\int_0^1e^{g_{10}x}|\hat{w}_x(x,t)|dx+2g_6e^{g_2}\hat{c}B^*(R)\nonumber\\
&&\times\left(3\alpha_{1,r_x}^2\left(\Lambda(t),R\right)+\kappa_4(R) \alpha_{r}^2\left(\Lambda(t),R\right)\right)\nonumber\\
&&-g_9|w_x(0,t)|+g_{11}\sup_{x\in[0,1]}\left|1-\hat{D}\pi(x,t)\right| e^{g_1}\nonumber\\
&&\times\alpha_{r}\left(\Lambda(t),R\right)+\left(\sup_{x\in[0,1]}\left|1-\hat{D}\pi(x,t)\right|\right)^2\left(8g_6\right.\nonumber\\
&&\times\!e^{g_2}(\hat{D}\!+\!1)^2\!\!+\!g_8e^{g_5}\kappa_3(R))\alpha_{r}^2\left(\Lambda(t),R\right).\label{wxwx}
\end{eqnarray}\setlength{\arraycolsep}{5pt}From the proof of Lemma \ref{brfo} (App. B) we have that $\alpha_r$ and $\alpha_{1,r_x}$ are continuously differentiable and hence locally Lipschitz with respect to their first argument. Using (\ref{222}) we can write $\alpha_r^2(s,R)\leq\alpha_{26}(R)\alpha_r(s,R)\leq  L_1(R)\alpha_{26}(R) s$ and $\alpha_{1,r_x}^2(s,R)\leq \alpha_{26}(R)\alpha_r(s,R)\leq L_2(R)\alpha_{26}(R)s$ for every bounded $s$, some increasing functions $L_i\in C\left(\mathbb{R}_+;\mathbb{R}_+\right)$, $i=1,2$, and some class $\mathcal{K}_{\infty}$ function $\alpha_{26}$. Choosing $g_6>\left(1+R\right)\left(\hat{D}+R\right)g_8e^{g_5}\kappa_3(R)$, $g_7>g_8e^{g_5}\kappa_3(R)$, $g_{11}>\left(1+R\right)\left(\hat{D}+R\right)\left(4g_6e^{g_2}\left(\hat{D}+1\right)^2\right.$ $4{B^*}^2(R)\kappa_2(R)$ $+\left(g_7e^{g_4}+g_9e^{g_{10}}+g_8e^{g_5}+g_{12}e^{g_3}\right)$ $\left.\kappa_5(R)+\frac{M_4L^*(R)}{2\sqrt{M_1}}\right)$, $g_4>\frac{1}{g_7}\left(e^{g_2}{g_6}\left(\kappa_4(R)+\frac{1}{\hat{D}}\right)\right.$ $\left.2B^*(R)+1\right)$, $g_3=g_5=g_8=g_9=g_{10}=g\!=\!1$ and since from the proof of Lemma \ref{lemz} $\alpha^*(s)>s$ choosing $g_{12}>\frac{M_4L^*(R)}{2\hat{D}\sqrt{M_1}}$ we get 
\setlength{\arraycolsep}{0pt}\begin{eqnarray}
\dot{V}(t)&\leq&-\left(\frac{M_3}{2\sqrt{M_1}}-\hat{c}B\right)|X(t)|-g_{11}  \int_0^1|\tilde{u}(x,t)|\nonumber\\
&&\times dx-g_6\left(1-\hat{c}B_2\right)\int_0^1e^{g_2x}\tilde{u}_x(x,t)^2dx\nonumber\\
&&-\int_0^1\hat{w}_x(x,t)^2dx-\left(1-\hat{c}\left(B_3+B\right)\right)\nonumber\\
&&\times\int_0^1|\hat{w}_x(x,t)|dx-\left(g_{12}-\hat{c}B\right)\nonumber\\
&&\times\int_0^1\alpha^*\left(|\hat{w}(x,t)|\right)dx-\left(1-\hat{c}B_1\right)\nonumber\\
&&\times\int_0^1\hat{w}_{xx}(x,t)^2dx  ,\label{wxwx123}
\end{eqnarray}\setlength{\arraycolsep}{5pt}
where we used (\ref{deflambda}) and
\begin{eqnarray}
B\left(R\right)&=&2g_6e^{g_2}B^*(R)\alpha_{26}(R)\left(3L_2(R)+\kappa_4(R)L_1(R)\right)\nonumber\\
&&+32g_6e^{g_2}\left(\hat{D}+1\right)^2L_1(R)\alpha_{26}(R) 4{B^*}^2(R)\nonumber\\
&&2L_1(R)B^*(R) \left(g_{11}e^{g_1}\!+\!2\alpha_{26}(R)e{B^*}(R)\right)\\
B_1(R)&=&6{g_6e^{g_2}}{\hat{D}^{-2}}B^*(R)\\
B_2(R)&=&2B^*(R)\\
B_3(R)&=&2eg_{11}e^{g_1}{\hat{D}}^{-1}B^*(R).
\end{eqnarray}
Restricting $c^*=\min\left\{c_1^*,{c}_2^*\right\}$ and $c_1$ such that $\hat{c}$ in (\ref{tosp})
 satisfies $\hat{c}\!<\!\min\left\{R,\hat{c}_1\right\}$, with $\hat{c}_1\!\max\left\{B_2,B_1,B_3+B\right\}\!\leq\!\frac{1}{2}\!\min\!\left\{\frac{M_3}{2\sqrt{M_1}},g_{12},1\right\}$, we arrive at $\dot{V}(t)\leq -\lambda V(t)$, with $\lambda=\frac{1}{2}\min\left\{\frac{M_3}{2\sqrt{M_1}},\right.$ $\left.2g_{11},g_6,1,g_{12}\right\}$. With the comparison principle (Lemma 3.4 in \cite{khalil}) we get (\ref{Afth}).

\section{Proof of Lemma \ref{lemag}}
\label{appp}

Using (\ref{issn1}), Lemma \ref{lemmauw} (App. B), the fact that for all $x\in[0,1]$, $\left|\tilde{u}(x,t)\right|\leq\int_0^1\left|\tilde{u}_x(x,t)\right|dx$ (which follows from (\ref{obs3non})), the Cauchy-Schwarz inequality and some routine class $\mathcal{K}$ calculations the proof is immediate.

\section{Proof of Lemma \ref{lem3}}
\label{lin1ap}
Taking the derivative of $V_{\rm L}$ along (\ref{sys2tr})--(\ref{sys2ntr}), (\ref{til1})--(\ref{til2}) and (\ref{trr1})--(\ref{sys2ntr11}) and using integration by parts, we get that
\begin{eqnarray}
\dot{V}_{\rm L}(t)&\leq&-\left|X(t)\right|^2\lambda_{\min}(Q)+2X(t)^TPB\hat{w}(0,t)\nonumber\\
&&+2X(t)^T PB\tilde{u}(0,t)-b_1b\int_0^1e^{bx}\pi(x,t)\tilde{u}(x,t)^2\nonumber\\
&&\times dx-b_1\pi(0,t)\tilde{u}(0,t)^2-b_1\int_0^1e^{bx}\pi_x(x,t)\nonumber\\
&&\times\tilde{u}(x,t)^2dx -2b_1\int_0^1e^{bx}\tilde{u}(x,t)\left({1}-\hat{D}\pi(x,t)\right)\nonumber\\
&&\times  r(x,t)dx+b_2\left(-\hat{w}(0,t)^2-\int_0^1\hat{w}(x,t)^2dx\right.\nonumber\\
&&\left.+2\hat{D}^2|K|^2e^{2|A|\hat{D}}\times |B|^2\tilde{u}(0,t)^2-\hat{w}_x(0,t)^2\right)\nonumber\\
&&-b_2\int_0^1\hat{w}_x(x,t)^2dx-2b_2\hat{D}\int_0^1(1+x)\nonumber\\
&&\times\hat{w}(x,t)Ke^{A\hat{D}x}B\tilde{u}(0,t)dx-2b_2\nonumber\\
&&\times\int_0^1\!(1+x)\hat{w}_x(x,t)\hat{D}^2\!Ke^{A\hat{D}x}\!AB\tilde{u}(0,t)dx.\label{kia}
 \end{eqnarray}
With similar calculations as in \cite{krstic varying} and with Lemma \ref{lemena} by choosing $b>(1-\pi_1^{**})\max\left\{1,\frac{1}{\pi_1^{**}}\right\}$ we get that 
\begin{eqnarray}
b\pi(x,t)+\pi_x(x,t)\geq\pi_0^{**}\beta^*,\label{axam}
\end{eqnarray}
 where $\beta^*=\min\left\{b-1+\pi_1^{**},(b+1)\pi_1^{**}-1\right\}>0$. 
 Since $\pi(x,t)$ is linear in $x$, it takes its maximum value either at $x=0$ or at $x=1$, and hence, 
\begin{eqnarray}
\left({1}-\hat{D}\pi(x,t)\right)&\leq&\max\left\{\left|{1}-\hat{D}\pi(0,t)\right|,\left|{1}-\hat{D}\pi(1,t)\right|\right\}\nonumber\\
&\leq&M_{2,\rm L}\gamma(t),\label{ggg1}
\end{eqnarray}
where 
\begin{eqnarray}
M_{2,\rm L}=\pi(0,t)+\pi(1,t)\leq \frac{1+\sup_{t\geq t_0}{\dot{\sigma}(t)}}{\inf_{t\geq t_0}{(\sigma(t)-t)}},\label{defm2}
\end{eqnarray}
 and $\gamma$ is defined in (\ref{starn}). We derive next a bound for $r(x,t)$ in terms of $X,\hat{w}$ and $\hat{w}_x$. Using (\ref{roo}) together with Young's and Cauchy-Schwarz's inequalities we get that
\begin{eqnarray}
\|r(t)\|^2&\leq& M_{1,\rm L} \left(|X(t)|^2+\|\hat{w}(t)\|^2+\|\hat{w}_x(t)\|^2\right)\label{boundr}\\
\|r(t)\|^2&=&\int_0^1r(x,t)^2dx,\label{newlo}\\
M_{1,\rm L}&=&{4}{\hat{D}^{-1}}+4\hat{D}\left|Ke^{(A+BK)\hat{D}}(A+BK)\right|^2+4\hat{D}\nonumber\\
&&\!\!\times |KB|^2\!+\!4\hat{D}\left|Ke^{(A+BK)\hat{D}}(A\!+\!\!BK)\hat{D}B\right|^2\!\!\!.\label{defm1}
\end{eqnarray}
Using (\ref{axam}), (\ref{ggg1}), (\ref{boundr}) and Young's inequality we get
\begin{eqnarray}
\dot{V}_{\rm L}(t)&\leq&-\frac{\lambda_{\min}(Q)}{4}\left|X(t)\right|^2 -b_1\pi_0^{**}\beta^*\int_0^1e^{bx}\tilde{u}(x,t)^2dx\nonumber\\
&&-\frac{b_2}{2}\int_0^1\hat{w}(x,t)^2dx-\frac{b_2}{2}\int_0^1\hat{w}_x(x,t)^2dx\nonumber\\
&&+\left(\frac{8|PB|}{\lambda_{\min}(Q)}-b_2\right)\hat{w}(0,t)^2+\left(\frac{8|PB|}{\lambda_{\min}(Q)}\right.\nonumber\\
&&\left.+8b_2\hat{D}^2|K|e^{|A|\hat{D}}|B|^2\left(2+|A|^2\right)-b_1\pi_0^{**}\right)\nonumber\\
&&\times\tilde{u}(0,t)^2+b_1M_{2,\rm L}(1+M_{1,\rm L})\gamma(t)\Xi(t),
 \end{eqnarray}
where
\begin{eqnarray}
\Xi(t)&=&|X(t)|^2+\int_0^1e^{bx}\tilde{u}(x,t)^2dx\nonumber\\
&&+\int_0^1\hat{w}(x,t)^2dx+\int_0^1\hat{w}_x(x,t)^2dx.
\end{eqnarray}
Choosing $b_1=\frac{8|PB|}{\lambda_{\min}(Q)}+8b_2\hat{D}^2|K|e^{|A|\hat{D}}|B|^2\left(2+|A|^2\right)$, $b_2=\frac{8|PB|}{\lambda_{\min}(Q)}$, and using the fact that
\begin{eqnarray}
M_{4,\rm L}\Xi(t)\leq V_{\rm L}(t)\leq M_{3,\rm L}\Xi(t)\label{bvv},
\end{eqnarray}
where
\begin{eqnarray}
M_{3,\rm L}&=&\lambda_{\max}(P)+b_1+2\hat{D}b_2\\
M_{4\rm, L}&=&\min\left\{\lambda_{\min}(P),b_1,\hat{D}b_2\right\},\label{defm4}
\end{eqnarray}
we get relation (\ref{vdottt}) with
\begin{eqnarray}
r_1&=&\frac{\min\left\{\lambda_{\min}(Q),4b_1\pi_0^{**}\beta^*,2{b_2}\right\}}{4M_{3,\rm L}}\label{defr1}\\
r_2&=&\frac{b_1M_{2,\rm L}(1+M_{1,\rm L})}{M_{4,\rm L}}.\label{defr2}
\end{eqnarray}

\section{Proof of Lemma \ref{lemfour}}
\label{fac1}

Consider first that $\gamma(t)\!=\!\left|\frac{1}{\pi(0,t)}-\hat{D}\right|$. Using (\ref{invp}), we get $\pi(x,t)=\frac{1+x\frac{\delta'\left(\sigma(t)\right)}{1-\delta'\left(\sigma(t)\right)}}{\hat{D}+\delta\left(\sigma(t)\right)}$, and hence, $\gamma(t)=\left|\delta\left(\sigma(t)\right)\right|$. Therefore, if $\delta$ satisfies (\ref{case1}) with $\delta_1r_2<{r_1}$, the lemma is proved. Assume next $\gamma(t)=\left|\frac{1}{\pi(1,t)}-\hat{D}\right|$. Thus, $\gamma(t)=\left|\delta(\sigma(t))\!\left(1\!-\!\delta'(\sigma(t))\right)\!-\!\hat{D}\delta'(\sigma(t))\!\right|$. With (\ref{ass1b}), (\ref{case1}) we get $\gamma(t)<\delta_1\left(1+\delta_1+\hat{D}\right)$. Hence,  the lemma is proved if $\delta_1r_2<\frac{r_1}{d}$, where $d={1}+\delta_1+\hat{D}$. Note that since one can choose $\pi_0^{**}\beta^*$ sufficiently large (by choosing a large $b$) such that $r_1$ in (\ref{defr1}) is independent of $\delta_1$ (or very large), and since from (\ref{defr2}) together with (\ref{defm1}), (\ref{defm2}), (\ref{defm4}) $r_2$ is bounded, one can always find a sufficiently small $\delta_1$ such that relation $\delta_1r_2<\frac{r_1}{d}$ is satisfied.

\section{Proof of Lemma \ref{lem4}}
\label{fac2}
Using relations (\ref{bck1}), (\ref{bckin}) together with Young's and Cauchy-Schwarz's inequalities we get that
\begin{eqnarray}
\|\hat{u}(t)\|^2&\leq&N_1\left(|X(t)|^2+\|\hat{w}(t)\|^2\right)\label{bbb1}\\
\|\hat{u}_x(t)\|^2&\leq&N_2\left(|X(t)|^2+\|\hat{w}(t)\|^2+\|\hat{w}_x(t)\|^2\right)\\
\|\hat{w}(t)\|^2&\leq&N_3\left(|X(t)|^2+\|\hat{u}(t)\|^2\right)\\
\|\hat{w}_x(t)\|^2&\leq&N_4\left(|X(t)|^2+\|\hat{u}(t)\|^2+\|\hat{u}_x(t)\|^2\right),\label{bbnn}
\end{eqnarray}
where $\|\cdot\|$ is defined in (\ref{newlo}) and
\begin{eqnarray}
N_1&=&3+3\left|Ke^{(A+BK)\hat{D}}\right|^2+3\hat{D}^2\left|Ke^{(A+BK)\hat{D}}B\right|^2\\
N_2&=&4+4\left|Ke^{(A+BK)\hat{D}}(A+BK)\hat{D}\right|^2+4\hat{D}^2|KB|^2\nonumber\\
&&+4\hat{D}^2\left|Ke^{(A+BK)\hat{D}}\right|^2|B|^2\left(1+\hat{D}|A+BK|\right)\\
N_3&=&3+3\left|Ke^{A\hat{D}}\right|^2+3\hat{D}^2\left|Ke^{A\hat{D}}B\right|^2\\
N_4&=&4+4\left|Ke^{A\hat{D}}A\hat{D}\right|^2+4\hat{D}^2|KB|^2\nonumber\\
&&+4\hat{D}^2\left|Ke^{A\hat{D}}B\right|^2+4\hat{D}^2\left|Ke^{A\hat{D}}A\hat{D}B\right|^2.
\end{eqnarray}
Using (\ref{bvv}), (\ref{bbb1})--(\ref{bbnn})  the lemma is proved with $M_{5,\rm L}=\frac{1}{ M_{3,\rm L}}\left(2e^{b} + N_3+N_4+1\right)$, $M_{6,\rm L}= \frac{3+3N_1 +N_2}{M_{4,\rm L}}$.

\end{document}